\documentclass[12pt]{article}
\usepackage{paperpkg}
\usepackage[margin=1in]{geometry}

\begin{document}

\title{An odd analog of Plamenevskaya's invariant of transverse knots}
\author{Gabriel Montes de Oca}

\maketitle

\abstract{Plamenevskaya defined an invariant of transverse links as a distinguished class in the even Khovanov homology of a link.
We define an analog of Plamenevskaya's invariant in the odd Khovanov homology of Ozsv\'ath, Rasmussen, and Szab\'o. 
We show that the analog is also an invariant of transverse links and has similar properties to Plamenevskaya's invariant. 
We also show that the analog invariant can be identified with an equivalent invariant in the reduced odd Khovanov homology. 
We demonstrate computations of the invariant on various transverse knot pairs with the same topological knot type and self-linking number.}
\tableofcontents

\newpage

\section{Introduction}
\label{sec-int}
\subsection{Background}
A \emph{transverse link} is a link that is everywhere transverse to the standard contact structure $(\R^3,\xi_{std})$. 
There are two ``classical invariants'' for transverse links: the smooth link type and the self-linking number (see \cite{Etn3} for the definition of contact structures and other basic facts about transverse links). 
In the early 2000s, Etnyre-Honda~\cite[Theorem~1.7]{EH} and Birman-Menasco~\cite[Theorem~3]{BM2} found the first examples of pairs of transverse links that had the same classical invariants but were not isotopic as transverse links. 
Topological links that have distinct transverse representatives with the same classical invariants are called \emph{transversely non-simple}.

Every link can be represented as the closure of some braid, and every braid can be associated to a transverse link;
 conversely, every transverse link is transversely isotopic to a closed braid~\cite[Th\'eor\`em~8]{Ben}. 
The Markov theorem gives conditions under which two braids have closures that are isotopic as links~\cite[theorem on p.~75]{Mar}. 
There is a transverse version of the Markov theorem that gives conditions under which two braids have closures that are isotopic as transverse links~\cite[Theorem~1]{Wri}~\cite[theorem on p.~1]{OS}.

At the turn of the millennium, Khovanov defined a categorification of the Jones polynomial, \emph{Khovanov homology}~\cite{Kho}. 
Khovanov homology $Kh(L)$ is the homology of a bigraded chain complex that is computed from the hypercube of resolutions of the link diagram. 
The graded Euler characteristic is a normalization of the Jones polynomial of the link. 

There are a number of transverse invariants coming from modern techniques in knot theory, particularly gauge theory and holomorphic curves. 
These include invariants in Heegaard Floer homology~\cite{LOSSz,OSzT,Kan}, monopole Floer homology~\cite{BS}, and knot contact homology~\cite{Ng}. 

In~\cite{Pla} Plamenevskaya identified a distinguished element in Khovanov homology that is an invariant of transverse links. 
It is not known to be effective. 
That is, there is no known pair of transverse links that have the same classical invariants but are distinguished by Plamenevskaya's invariant.
Lipshitz, Ng, and Sarkar further studied and refined this invariant and showed it is the same for pairs of transverse links related by negative flypes and pairs related by $SZ$ moves~\cite[Theorem~4.15]{LNS}. 

With $\mathbb Z/2$-coefficients, Ozsv\'ath and Szab\'o constructed a spectral sequence from Khovanov homology to the Heegaard Floer homology group $\widehat{\mathit{HF}}$ of the branched double cover~\cite[Theorem~1.1]{OSz}.
In attempting to lift this spectral sequence to $\Z$ coefficients, Ozsv\'ath, Ras\-mussen, and Szab\'o defined a variant of Khovanov homology, called odd Khovanov homology, and conjectured there is a spectral sequence from it to $\widehat{HF}(\Sigma(K))$ with $\mathbb Z$-coefficients~\cite[Conjecture~1.9]{ORSz}. 
In spite of a similar definition, odd Khovanov homology $Kh'(L)$ has different properties from even Khovanov homology.
The unreduced and reduced odd Khovanov homologies have a simpler relationship than in the even case. 
Shumakovitch showed there is more torsion in reduced odd Khovanov homology than reduced even Khovanov homology for small knots~\cite[Subsection~3.A]{Shu1}.

Although odd Khovanov homology is a strong invariant, it does not distinguish knots related by a Conway mutation~\cite[Theorem~1]{Blo}.
There is no known analogue of the Lee spectral sequence, an object defined in~\cite{Lee2,Lee1} (cf.~\cite{Ras}). 
This last observation is notable given that there is a close relationship between Pla\-me\-nev\-ska\-ya's invariant and the Lee spectral sequence~\cite[Theorem~4.2]{LNS}.

\subsection{Notation}
\label{ssec-not}
In general, we will denote links with some variation of $L$ and their diagrams as $D$. 
We will refer to the set of crossings in a diagram as $\mathcal X$, with the number of crossings as $n$. 
When we need to count the positive and negative crossings, we will use $n_+$ and $n_-$ respectively.

As this paper defines an invariant of transverse links, we will be considering diagrams that are manifestly the closure of some braid $B$. 
The number of strands in such a braid will be denoted by $b$. 
The self-linking number of the tranverse link obtained from the closure of $B$ is
	\[sl(L) = -b+n_+-n_-. \]

The even and odd Khovanov homologies are each constructed from the cube of resolutions of a diagram. 
Each resolution $\alpha:\mathcal X\too\{0,1\}$ corresponds to a resolution diagram (often metonymically referred to simply as the resolution) from a tranverse link diagram by replacing each crossing with the 0- or 1-smoothing as shown in Figure~\ref{fig-sm1}, depending on $\alpha(x)$. 
\begin{figure}
\caption{0- and 1-smoothings of a crossing.}
\label{fig-sm1}
	\begin{center}\begin{tikzpicture}
		\nxp{0}{2}
		\hsp{-1.5}{0}
		\vsp{1.5}{0}
		
		\draw [->]  (0,2.2) -- (-.5,1.4);
		\draw [->] (1,2.2) -- (1.5,1.4);
		
		\draw (-1,0) node {0-smoothing};
		\draw (2,0) node {1-smoothing};
	\end{tikzpicture}\end{center}
\end{figure}
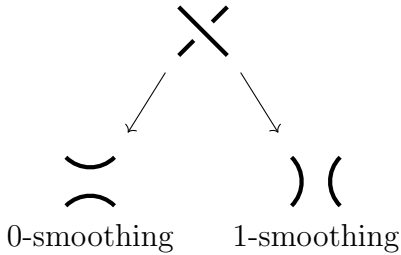
Given a resolution $\alpha$, its length is defined
	\[|\alpha|=\sum_{x\in\mathcal X} \alpha(x).\]

The circles $\{a_1, \cdots, a_k\}$ that make up the resolution diagram of $\alpha$ in the odd Khovanov homology, are associated to the set of generators $\{v_1,\dots,v_k\}$, with $a_i$ corresponding to $v_i$ for each $i$. 
The resolution $\alpha$ thus corresponds to the exterior algebra on all such generators
	\[\Lambda^*V_\alpha = \Lambda^*\langle v_1,\dots,v_k\rangle.\]

Together with a sign assignment (see Section~\ref{ssec-sa}), the edges of the cube are associated to linear maps. 
If the resolutions at each end of the edge can be connected by the cobordism in which the circles $a_0$ and $a_1$ in the source merge into $a$ as the target, then the map on the exterior algebra is given by the generators associated to $a_0$ and $a_1$ both mapping to the generator associated to $a$. 
On the other hand, if the resolutions on either side of an edge can be connected by the cobordism in which the circle $a$ in the source splits into $a_0$ and $a_1$ in the target, then if $v_0$ and $v_1$ are the generators associated to $a_0$ and $a_1$ respectively, the map is defined (up to a sign) by
	\[\omega \mapsto (v_0-v_1)\wedge \omega.\]

For a link $L$ with diagram $D$, both the cube and its associated chain complex will be denoted by $C(D)$.
The homological grading for elements corresponding to a resolution $\alpha$ is defined
	\[r = n_- -|\alpha|,\]
and the quantum grading of an element in $\Lambda^kV_\alpha$ is
	\[Q = (\dim V_\alpha) - 2k + n_+ - 2n_- + |\alpha|.\]
	
The even Khovanov homology is denoted $Kh(L)$, and the odd Khovanov homology is denoted $Kh'(L)$.

\subsection{The Result}
\label{sec-result}
In this paper, for a transverse link $L$ with diagram $D$, we construct an analog of Plamenevskaya's invariant for transverse links in odd Khovanov homology, $\psi(D)$. 
\begin{thm}
\label{thm-result}
The element $\psi(D)\in Kh'(L)$ is a transverse link invariant, which is well defined up to a sign.
\end{thm}
This theorem is restated more precisely as Theorem~\ref{thm-inv} in Section~\ref{sec-inv}. 
In Proposition~\ref{prop-negstab}, we show that the odd Plamenevskaya invariant of the negative stabilization of another transverse link is zero. 
In Proposition~\ref{prop-posx}, we show that
if $L'$ can be obtained from $L$ by replacing a single positive crossing with a 0-smoothing, then the invariants of each are related by the associated homomorphism $Kh'(L)\too Kh'(L')$. 
Unlike even Khovanov homology~\cite[Theorem~1]{Kho2}, it is not known if odd Khovanov homology is natural (cf.~\cite{Put}), so this identification is weaker than the analogous identification in the even case~\cite[Theorem~4]{Pla}.

\subsection{Organization}
\label{sec-organization}
The odd analog of Plamenevskaya's invariant is defined in Section~\ref{sec-inv}, and, using the transverse Markov theorem, we prove it to be invariant in Theorem~\ref{thm-inv}. 

In Section~\ref{sec-rh}, we investigate the reduced odd Khovanov homology. 
There, we define a reduced version of the invariant, and in Proposition~\ref{prop-opiinrokh}, prove that the unreduced invariant can be identified with the reduced invariant via the isomorphism between full odd Khovanov homology and reduced odd Khovanov homology. 

In Section~\ref{sec-prop}, we investigate the odd invariant's properties analogous to those of the even Plamenevskaya invariant. 
In Section~\ref{sec-comp}, we discuss the author's computer program for studying the invariant and observations made using it.

\section{The Invariant}
\label{sec-inv}

\subsection{Definition of the Invariant}
\label{ssec-definv}
\begin{defn}
\label{defn-opi}
Let $L$ be a transverse link and $D$ be a braid diagram of $L$. 
In the resolution cube associated to $D$, let 
 $\alpha'$ be the unique resolution where the braid representation is separated into $b$ parallel bands. 
This resolution corresponds to the vector space $\Lambda^*V_{\alpha'}$, where $V_{\alpha'}=\langle v_1,\dots,v_b\rangle$.
We define $\psi(D)$ first on the level of the chain complex to be a generator of $\Lambda^bV_{\alpha'}$,
	\[\psi(D) := v_1\wedge \cdots \wedge v_b.\]
\end{defn}

From the braid representation, it is easy to see that this resolution is the one in which there is a 0-smoothing for every positive crossing and a 1-smoothing for every negative crossing.

\subsection{The Invariant as Seen in Homology}
\label{ssec-invhom}
\begin{prop} 
\label{prop-tpsic}
$\psi(D)$ is a cycle.
\end{prop}

\begin{proof}
There are two cases. 
If the resolution corresponding to the vertex in which $\psi(D)$ resides is one with a 1-smoothing at every crossing (that is, every crossing in $D$ is a negative crossing), then the next vector space in the chain complex is the 0 vector space, so the differential from the vector space containing $\psi(D)$ is the zero map. 
Thus every element of $\Lambda^*V_{\alpha'}$ is trivially a cycle, $\psi(D)$ included.

In the second case, there is at least one 0-smoothing in the corresponding resolution. 
The differential that maps out of $\Lambda^*V_{\alpha'}$ in this case is a sum of maps, each corresponding to a merge cobordism, one for each 0-smoothing. 
This is because at each  0-smoothing, the parallel strings on either side of the smoothing merge into a single circle after becoming the 1-smoothing. See Figure~\ref{fig-cyclemerge}. 
We will show that for any one of these maps $\psi(D)$ is mapped to 0, thus $d(\psi(D))=0$ and is therefore a cycle. 
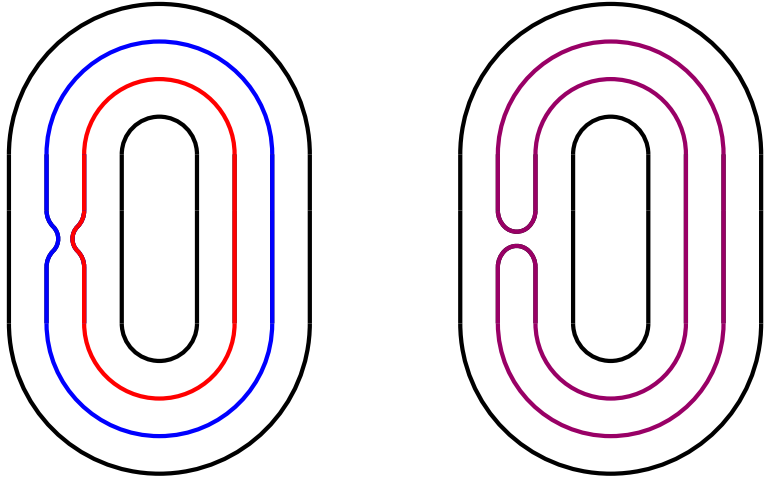
\begin{figure}
	\begin{center}\begin{tikzpicture}[scale=0.5]
		\topcap{0}{0}{4}

		\topcap{3}{0}{1}
		\threeres{4}{0}{0}{5}{2}{5}{0}{0}{0}
		\threebraid{4}{5}{0}{5}{5}{5}
		\botcap{0}{2}{4}
		\botcap{3}{2}{1}
		\topcap{12}{0}{4}
		\topcap{15}{0}{1}
		\threeres{4}{12}{0}{5}{2}{5}{1}{1}{1}
		\threebraid{4}{17}{0}{5}{5}{5}
		\botcap{12}{2}{4}
		\botcap{15}{2}{1}
		
	\color{blue}
		\topcap{1}{0}{3}
		\botcap{1}{2}{3}
		\draw[ultra thick] (7,1.5) -- (7,-3);
		\threeres{2}{1}{0}{5}{1}{5}{0}{0}{0}
	
	\color{red}
		\topcap{2}{0}{2}
		\botcap{2}{2}{2}
		\draw[ultra thick] (6,1.5) -- (6,-3) (2,1.5)--(2,0) (2,-1.5)--(2,-3);
		\lc{1}{1}
		
	\color{blue!40!red}
		\topcap{13}{0}{3}
		\topcap{14}{0}{2}
		\botcap{13}{2}{3}
		\botcap{14}{2}{2}
		\draw[ultra thick] (18,1.5)--(18,-3) (19,1.5)--(19,-3);
		\threeres{2}{13}{0}{5}{1}{5}{1}{1}{1}
	
	\end{tikzpicture}\end{center}
\caption[Edge maps out of the invariant's resolution.]{\figbld{Edge maps out of the invariant's resolution.} The diagrams of the 0-smoothing (left) and 1-smoothing (right) of a single positive braid crossing. The two circles (blue and red) in the 0-smoothing on the left merge into a single circle (purple) on the right.} 
\label{fig-cyclemerge}
\end{figure}

If the merging components in the diagram are  $a_{i_0}$ and $a_{i_1}$, corresponding to generators  $v_{i_0}$ and $v_{i_1}$ resp., the merge map on the vector spaces is induced by the quotient map,
	\[q: V_{\alpha'} \too V_{\alpha'}/(v_{i_0}-v_{i_1}).\]
The image of the quotient map $V_{\alpha'}/(v_{i_0}-v_{i_1})$ is isomorphic to the vector space generated by the elements corresponding to the components in the target resolution $\langle v_1,\dots,(v_{i_0}\sim v_{i_1}),\dots,v_b\rangle$ since there will be one fewer component there after the merge. 
Since $q(v_{i_0})=q(v_{i_1})$, then $\tilde{q}(\psi(D))=0$ under the induced map because two of the factors in the wedge product map to the same vector.
\end{proof}

Thus, $\psi(D)$ defines an element of homology. We will abuse our notation and refer to both the cycle and its class in homology by $\psi(D)$.

\begin{prop} 
\label{prop-opigrading}
The distinguished element $\psi(D)$ is in $Kh'_{0,sl(L)}(L)$. 
\end{prop}

\begin{proof} 
Let $\psi(D)$ be the distinguished element of the homology defined above corresponding to the diagram $D$. 
In the chain complex, $\psi(D)$ is an element of $\Lambda^*V_{\alpha'}$ where $V_{\alpha'}=\langle v_1,\dots,v_b\rangle$ is the vector space with a generator associated to each of the $b$ parallel strands. 
In particular, the $Q$ grading on $\Lambda^bV_{\alpha'}$ is given by
	\begin{align*}
	Q	&=(\dim V_{\alpha'})- 2b+n_+-2n_- + |\alpha'|,
	\intertext{and because of the choice of $\alpha'$, the number of 1-smoothings is the number of negative crossings in $D$. Thus, we have $|\alpha'|=n_-$, so}
	Q	&=-b+n_+-n_-\\
		&=sl(L).
	\end{align*}
	
For each resolution $\alpha$, the homological grading $r$ on $V_\alpha$ is defined such that $|\alpha|=r+n_-$. 
Thus, since $|\alpha'|=n_-$, then $r=0$. Therefore, $\psi(D)\in Kh'_{0,sl(L)}(L)$.
\end{proof}

\subsection{Invariance}
\label{ssec-inv}
In this subsection, we will show that $\psi(D)$ is an invariant of transverse links. 
To do this, we rely upon the transverse Markov theorem\----Theorem~\ref{thm-tmt}, below\----and show that $\psi(D)$ is invariant under braid group relations and positive braid stabilizations.

\begin{thm}[Transverse Markov Theorem, Wrinkle and Orevkov-Shevchushin, {\cite[Theorem 1]{Wri}},{\cite[theorem on p.1]{OS}}]
\label{thm-tmt}

Two tranverse links $L$ and $L'$ are transversely isotopic if and only if they are related by a finite number of the following moves:
\begin{itemize}
\item braid group relations,
\item braid conjugations,
\item positive braid stabilizations and destabilizations.
\end{itemize}
\end{thm}

It is trivially the case that $\psi(D)$ is unchanged by braid conjugations since two braids related by braid conjugation have closures whose diagrams are isotopic in the plane. 
Thus, their chain complexes are also canonically isomorphic, and that isomorphism clearly identifies $\psi(D)$. 
Positive braid stabilization and destabilization  corresponds to a Reidemeister move of type I that introduces or removes a single positive crossing. 
We refer to such a move as a transverse type I Reidemeister move, and we prove invariance of $\psi(D)$ under this move in Proposition~\ref{prop-r1}. 
The braid group moves can be generated from Reidemeister moves of types II and III. 
We show the invariance of $\psi(D)$ under these moves in Propositions~\ref{prop-r2}~and~\ref{prop-r3}.

\begin{prop} 
\label{prop-r1}
Let $D$ and $\hat D$ be two braid diagrams for a transverse link $L$ related by a single transverse type I Reidemeister move (R1), where $\hat D$ is the diagram with the additional positive crossing. 
There is a quasi-isomorphism $\rho:C(D)\too C(\hat D)$ such that
$\rho(\psi(D))=\psi(\hat D)$.
\end{prop}

\begin{proof}
Let $D$ and $\hat D$ be as described above. 
Focusing on the additional positive crossing, $\hat D$ has two associated diagrams: $D_0$ (resp.  $D_1$) where the 0-smoothing  (resp. 1-smoothing) resolves the additional crossing. 
The resolution cube of  $D_1$ is isotopic to the resolution cube of $D$ at corresponding vertices, thus there is a natural identification of $\psi(D)$ and $\psi(D_1)$. 
On the other hand, the resolution cube of $D_0$ is isotopic to the resolution cube of $D\sqcup a_0$. 
See Figure~\ref{fig-r1}. 
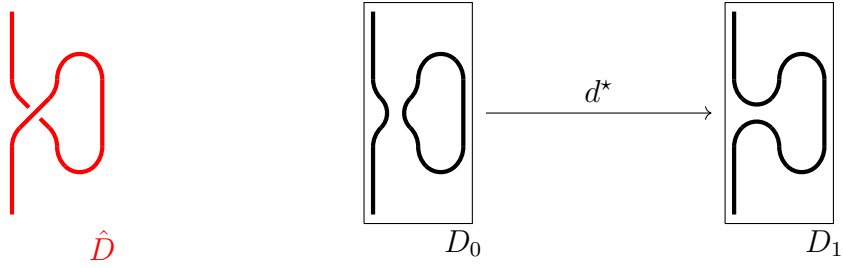
\begin{figure}
	\begin{center}\begin{tikzpicture}[scale=0.6]
		\color{red}
		\braidline{3}{2}{0}{0}{4}
		\braidline{3}{1}{0}{1}{1}
		\braidline{3}{2}{0}{2}{5}

		\draw (2,-3.1) node[below]{$\hat D$};

		\color{black}
		\braidline{3}{2}{8}{0}{4}
		\braidline{3}{1}{8}{1}{3}
		\braidline{3}{2}{8}{2}{5}

		\braidline{3}{2}{16}{0}{4}
		\braidline{3}{1}{16}{1}{2}
		\braidline{3}{2}{16}{2}{5}
		\def\dx{8}
		\def\dy{-3}
		\long\def\diagbox#1#2#3{\draw (#1*\dx,#2*\dy) ++(-.2,-3.2) -- ++(2.4,0) -- ++(0,4.9) -- ++(-2.4,0) -- ++(0,-4.9);
		\draw (#1*\dx,#2*\dy) ++(2,-3.1) node[below]{$#3$};}
		\draw (13,-.75) node[above] {$d^\star$};

		\diagbox{1}{0}{D_0}
		\diagbox{2}{0}{D_1}

		\draw[->] (10.5,-.75) -- (15.5,-.75);
	\end{tikzpicture}\end{center}
\caption[The Resolution Cube for $\hat D$: diagram after a transverse RI move.]{\figbld{The Resolution Cube for $\hat D$: diagram after a transverse RI move.}}
\label{fig-r1}
\end{figure}
	
In their respective chain complexes, we will use the same generators for equivalent circles. 
In particular, in each vertex in the resolution cube, we will always label the circle to which $a_0$ attaches as $a_1$. 
We associate $a_0$ to the generator $v_0$, and $a_1$ to $v_1$. 
We also note, by using the same generators for equivalent circles, since
	\[C(D_1)\oplus (v_0\wedge C(D_1))\cong
	C(D)\oplus(v_0\wedge C(D))\cong 
	C(D\sqcup a_0) \cong 
	C(D_0), \]
there is a natural inclusion of the first summand, $\imath:C(D_1) \lhook\joinrel\longrightarrow C(D_0)$.
If we define $w:C(D_0)\too C(D_0)$ by $w(\omega)=(v_0-v_1)\wedge\omega$, the composition $w\circ\imath:C(D_1)\too C(D_0)$ induces an isomorphism of chain complexes between $C(D_1)$ and $(v_0-v_1)\wedge C(D_0)$. 
In particular, if we let $b$ be the braid index of $D$ and thus the braid index of $D_1$, we have
	\begin{align*}
	w\circ\imath(\psi(D_1)) &= w\circ\imath(v_1\wedge\cdots\wedge v_b)\\
	&=(v_0-v_1)\wedge(v_1\wedge\cdots\wedge v_b)\\
	&=v_0\wedge v_1\wedge\cdots\wedge v_b.
	\end{align*}

Now, we consider the chain map $d^\star: C(D_0) \too C(D_1)$, which is the map induced by the cobordism merging $a_0$ and $a_1$. 
This map is the quotient map given by the identification of $v_0$ with $v_1$. 
With this setup, the chain complex $C(\hat D)$ is isomorphic to the mapping cone $\cone(d^\star)$. 

On the chain complex level, $\ker d^\star$ is isomorphic to $(v_0-v_1)\wedge C(D_0)$. 
Since $d^\star$ is surjective, it follows that $C(\hat D)\cong \cone(d^\star)$ is quasi-isomorphic to $\ker d^*\cong (v_0-v_1)\wedge C(D_0)$ via $j:\ker d^\star\lhook\joinrel\longrightarrow \cone(d^\star)$. Thus, we have a quasi-isomorphism
	\[\overline\jmath :(v_0-v_1)\wedge C(D_0)\cong\ker d^\star\too C(\hat D),\]
with 
	\[\overline\jmath(v_0\wedge v_1\wedge \cdots\wedge v_b)=v_0\wedge v_1\wedge\cdots\wedge v_b=\psi(\hat D).\]
	
Letting $\rho:C(D)\too C(\hat D)$ be the composition $\bar\jmath$ after $w\circ\imath$, it follows that $\rho$ is a quasi-isomorphism such that
	\[\rho(\psi(D))=\psi(\hat D).\qedhere\]
\end{proof}

\begin{prop}
\label{prop-r2}
Let $D$ and $\hat D$ be two braid diagrams of a transverse link $L$ related by a single type II Reidemeister move (R2), where $D$ is the diagram with more crossings. 
There is a quasi-isomorphism $\rho:C(D)\too C(\hat D)$ such that $\rho(\psi(D))=\psi(\hat D)$. 
\end{prop}

\begin{proof}
Let $D$ and $\hat D$ be as described above. 
The resolution cube for $D$ is illustrated in Figure~\ref{fig-r2}. 
We note $C(D)$ can be represented by the diagram below as a mapping cone of a map between two mapping cones.
	\begin{center}\begin{tikzcd}
		& C(D_{01}) \arrow[rd, "d^{\star1}"] & \\
		C(D_{00}) \arrow[ru, "d^{0\star}"] \arrow[rd, "\wedge(v_2-v_1)"'] && C(D_{11})\\
		& C(D_{10}) \arrow[ru, "v_2\sim v_3"']
	\end{tikzcd}\end{center}
By the arrangement in  Figure~\ref{fig-r2}, it follows that $\psi(D)\in C(D_{01})$.
\begin{figure}
	\begin{center}\begin{tikzpicture}[scale=\rsc]
		\def\dx{4}
		\def\dy{-3}
		\long\def\diagbox#1#2#3{\draw (#1*\dx,#2*\dy) ++(-.2,-1.7) -- ++(1.4,0) -- ++(0,3.4) -- ++(-1.4,0) -- ++(0,-3.4);
		\draw (#1*\dx,#2*\dy) ++(1,-1.6) node[below]{$#3$};}
		\color{red}
		\twobraid{2}{-4}{0}{1}{-1}
		\draw (-3,-1.6) node[below] {$D$};
		\color{white}
		\draw (13,0) node {.};
		\color{black}
		\tworesdiag{2}{0}{0}{1}{-1}
		\diagbox{0}{1}{D_{00}}
		\diagbox{1}{0}{D_{01}}
		\diagbox{1}{2}{D_{10}}
		\diagbox{2}{1}{D_{11}}

		\draw[->] (1.5,-2)--(3.5,-1);
		\draw[->] (1.5,-4)--(3.5,-5);
		\draw[->] (5.5,-1)--(7.5,-2);
		\draw[->] (5.5,-5)--(7.5,-4);
	\end{tikzpicture}\end{center}
\caption[The resolution cube for the diagram after an RII move.]{\figbld{The resolution cube for $D$: the diagram after an RII move.}}
\label{fig-r2}
\end{figure}
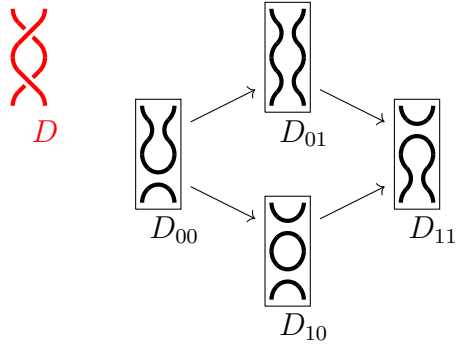

We let $X\subset C(D_{10})$ be the kernel of the contraction with $v^*_2$, the dual of the generator associated to the disjoint circle in $D_{10}$. 
Note, since $X$ and $C(D_{11})$ are isomorphic via the quotient map $v_2\sim v_3$, it follows that the subquotient complex corresponding to the isomorphism's mapping cone $A=$
	\begin{center}\begin{tikzcd}
		&&&&& C(D_{11})\\
		&&&&X \arrow[ru, "v_2\sim v_3"']
	\end{tikzcd}\end{center}
is acyclic. 
Thus, $C(D)/A$, represented in the diagram below, is quasi-isomorphic to $C(D)$. 
	\begin{center}\begin{tikzcd}
		& C(D_{01})& \\
		C(D_{00}) \arrow[ru, "d^{0,\star}"] \arrow[rd, "\wedge(v_2-v_1)"'] &&&& \\
		& C(D_{10})/X. 
	\end{tikzcd}\end{center}

Let $q:C(D)\too C(D)/A$ be that quotient map, which is a quasi-iso\-mor\-phism. Since $\psi(D)\in C(D_{01})$ is the sole representative in its equivalence class in the quotient, then $q(\psi(D))=\psi(D)$.
Furthermore, since $C(D_{00})$ and $C(D_{10})/X$ are isomorphic under the map generated from $1\mapsto (v_1-v_2)$, then $(C(D)/A)/C(D_{01})$ is acyclic. 
Thus $C(D_{01})$ is quasi-isomorphic to $C(D)/A$, and the map is the natural inclusion map. 
Under this map, we have that $\psi(D)\in C(D_{01})\subset C(D)$ is unchanged. Finally, since there is a trivial isomorphism between $C(\hat D)$ and $C(D_{01})$, it follows that  $\psi(\hat D) = \rho(\psi(D))$, where $\rho$ is the quasi-isomorphism between $C(D)$ and $C(\hat D)$ obtained from the compositions of the quotient maps above and the trivial isomorphism from $C(D_{01})$ to $C(\hat D)$.
\end{proof}

\begin{prop}
\label{prop-r3}
Let $D$ and $\hat D$ be two braid diagrams of a transverse link $L$ related by a single type III Reidemeister move (R3). 
There is a chain complex $C$ and quasi-isomorphisms,
	\[\rho:C(D)\too C\qquad\text{and}\qquad\hat\rho: C(\hat D)\too C,\]
such that $\rho(\psi(D))=\hat\rho(\psi(\hat D))$. 
\end{prop}

\begin{proof}
Let $D$ and $\hat D$ be two link diagrams that are related by a single Reidemeister move of type 3, (R3).
Focusing on the three crossings involved in the (R3) move, we can represent $C(D)$ via the cube depicted in Figure~\ref{fig-r3D}. 
From the blue map in Figure~\ref{fig-r3D}
	\[d^{0\star0}:C(D_{000}) \too C(D_{010})\]
we define
	\[\tilde{d}^{0\star0}: \omega\mapsto d^{0\star0}(\omega)\wedge v_0,\]
where $v_0$ is the generator associated to the sole circle entirely shown in $D_{010}$.
Thus, if we denote the complex from the mapping cone of $\tilde{d}^{0\star0}:C(D_{000})\too C(D_{010})\wedge v_0$ by $C(\tilde D_{0{\star}0})$, it follows that there is a quasi-isomorphism between $C(D)$ and $C(D)/C(\tilde D_{0{\star}0})$. 
We also note that this quasi-isomorphism is the identity map on parts of the cube uninvolved in the quotient, namely on $C(D_{111})$. 
\begin{figure}
	\begin{center}\begin{tikzpicture}[scale=\rsc]
		\def\dx{7}
		\def\dy{-7.5}
		\long\def\diagbox#1#2#3{\draw (#1*\dx,#2*\dy) ++(-.2,-3.2) -- ++(2.4,0) -- ++(0,4.8) -- ++(-2.4,0) -- ++(0,-4.8);
		\draw (#1*\dx,#2*\dy) ++(2,-3.1) node[below]{$#3$};}
		\color{red}
		\threebraid{3}{-3}{0}{-1}{-2}{-1}
		\draw (-1,-3) node[below]{D};
		\color{black}
		\threeresdiagtall{3}{0}{0}{-1}{-2}{-1}{1}
		\diagbox{0}{1}{D_{000}}
		\diagbox{1}{0}{D_{001}}
		\diagbox{1}{1}{D_{010}}
		\diagbox{1}{2}{D_{100}}
		\diagbox{2}{0}{D_{011}}
		\diagbox{2}{1}{D_{101}}
		\diagbox{2}{2}{D_{110}}
		\diagbox{3}{1}{D_{111}}
		\draw[->] (3,-6.5)--++(3,4);
		\draw[->,very thick, blue] (3, -8.25)--++(3,0);
		\draw[->] (3,-10)--++(3,-4);

		\draw[->] (10,-.75)--++(3,0); 
		\draw[->] (10,-2.5)--++(1.2,-1.6) ++(.6,-.8)--++(1.2,-1.6);
		\draw[->,very thick, dashed, green!50!black] (10,-6.5)--++(3,4);
		\draw[->,very thick, green!50!black] (10,-10)--++(3,-4);
		\draw[->] (10,-14)--++(1.2,1.6) ++(.6,.8)--++(1.2,1.6);
		\draw[->] (10,-15.75)--++(3,0);

		\draw[->] (17,-2.5)--++(3,-4);
		\draw[->] (17,-8.25)--++(3,0);
		\draw[->] (17,-14)--++(3,4);
	\end{tikzpicture}\end{center}
\caption[The resolution cube for the diagram before an RIII move.]{\figbld{The resolution cube for $D$: the diagram before an RIII move.}}
\label{fig-r3D}
\end{figure}

We define $\mathring C$ to be the complex from the mapping cone of the identification of $C(D_{010})/v_0$ with $C(D_{011})\cong C(D_{110})$, shown in Figure~\ref{fig-r3D0}. 
Since the map is an isomorphism, this complex is acyclic. 
It can be identified with either green arrow in $C(D)$ in Figure~\ref{fig-r3D}. 
There is a chain map $\Psi: \mathring C \too C(D)/C(\tilde D_{0{\star}0})$, given by the identification of $C(D_{010})/v_0$ in $\mathring C$ with $C(D_{010})/(C(D_{010})\wedge v_0)$ in $C(D)/C(\tilde D_{0{\star}0})$, and the map from the codomain in $\mathring C$ to the quotient via the diagonal identification of $C(D_{011})$ and $C(D_{110})$. 
\begin{figure}
	\begin{center}\begin{tikzpicture}[scale=\rsc]
		\draw (-6.5,-.75) node{$C\Bigg($};
		\draw (-1.2,-.75) node{$\Bigg)/ v_0$};
		\threeres{3}{-5}{0}{-1}{-2}{-1}{0}{1}{0}
		\draw[->] (0.6,-.75) -- (4,-.75);
		\draw (5.5,-.75) node{$C\Bigg($};
		\braidline{3}{1}{7}{0}{5}
		\braidline{2}{1}{8}{1}{6}
		\braidline{3}{1}{7}{2}{4}
		\draw (10,-.75) node{$\Bigg)$};
	\end{tikzpicture}\end{center}
\caption[The resolution cube for subquotient $\mathring C$.]{\figbld{The resolution cube for $\mathring C$.}}
\label{fig-r3D0}
\end{figure}

There is a further chain map $\Phi$ from $C(D)/C(\tilde D_{0{\star}0})$ to the complex from the diagram in Figure~\ref{fig-riiiDl}, which we call $C$. 
Up to signs, this map is given by identifying $C(D_{001})$ with $C(A)$, $C(D_{100})$ with $C(B)$, $C(D_{011})$ and $C(D_{110})$ with $C(\Gamma)$, $C(D_{111})$ with $C(\Delta)$, and the map from $C(D_{010})$ being trivial. 
It is important to note that these signs can be arranged so that they do not impact the mapping between $C(D_{111})$ and $C(E)$ or adjacent maps. 
So defined, these chain maps form a short exact sequence,
  	\[0\too\mathring C \overset{\Psi}{\too} C(D)/C(\tilde D_{0{\star}0})\overset{\Phi}{\too} C\too 0.\]
Since $\mathring C$ is acyclic, $\Phi$ is a quasi-isomorphism. 
Furthermore, we note that $\Phi$ is the identity on $C(D_{111})$. 
Thus, there is a quasi-isomorphism $\rho:C(D)\too C$, which is the identity when restricted to $C(D_{111})$. 
\begin{figure}
	\begin{center}\begin{tikzpicture}[scale=\rsc]
		\def\dx{7}
		\def\dy{-6}
		\long\def\diagbox#1#2#3{\draw (#1*\dx,#2*\dy) ++(-.2,-3.2) -- ++(2.4,0) -- ++(0,4.8) -- ++(-2.4,0) -- ++(0,-4.8);
		\draw (#1*\dx,#2*\dy) ++(2,-3.1) node[below]{$#3$};}
		\diagbox{1}{1/4}{A}
		\diagbox{1}{7/4}{B}
		\diagbox{2}{1/4}{\Gamma}
		\diagbox{2}{7/4}{\Delta}
		\diagbox{3}{1}{E}
		\threeres{3}{7}{1}{-1}{-2}{-1}{0}{0}{1}
		\threeres{3}{7}{7}{-1}{-2}{-1}{1}{0}{0}
		\threeres{3}{14}{1}{-1}{-2}{-1}{0}{1}{1}
		\threeres{3}{14}{7}{-1}{-2}{-1}{1}{0}{1}
		\threeres{3}{21}{4}{-1}{-2}{-1}{1}{1}{1}
		\draw[->] (10,-2)--(13,-2); 
		\draw[->] (10,-3.75)--(11.4,-6.55) (11.6,-6.95)--(13,-9.75);
		\draw[->] (10,-9.75)--(13,-3.75);
		\draw[->] (10,-11.5)--(13,-11.5);
		\draw[->] (17,-2)--(20,-5);
		\draw[->] (17,-11.5)--(20,-8.5);
	\end{tikzpicture}\end{center}
\caption[The resolution complex for reduced $D$.]{\figbld{The resolution complex for $C$: the reduced resolution complex of $D$.}}
\label{fig-riiiDl}
\end{figure}
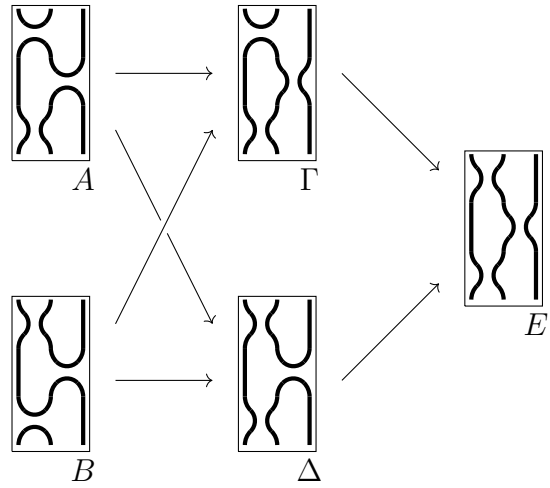

For clarity, we simplify the diagram of $C$ to the diagram in Figure~\ref{fig-r3D1simple}.
This will make the identification of $C$ with the contracted version of $C(\hat D)$ more visually obvious.
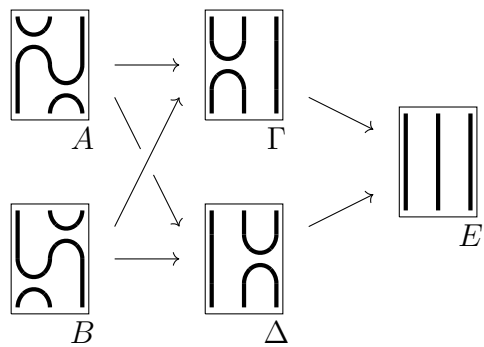
\begin{figure}
	\begin{center}\begin{tikzpicture}[scale=\rsc]
		\def\dx{6}
		\def\dy{-6}
		\long\def\diagbox#1#2#3{\draw (#1*\dx,#2*\dy) ++(-.2,-1.7) -- ++(2.4,0) -- ++(0,3.4) -- ++(-2.4,0) -- ++(0,-3.4);
		\draw (#1*\dx,#2*\dy) ++(2,-1.6) node[below]{$#3$};}

		\diagbox{0}{0}{A}
		\diagbox{0}{1}{B}
		\diagbox{1}{0}{\Gamma}
		\diagbox{1}{1}{\Delta}
		\diagbox{2}{.5}{E}
		\braidline{3}{1}{0}{0}{2}
		\braidline{3}{2}{0}{1}{2}

		\braidline{3}{2}{0}{4}{2}
		\braidline{3}{1}{0}{5}{2}

		\begin{scope}[yshift = -.75cm]
		\braidline{3}{1}{6}{0}{2}
		\end{scope}
		\draw[ultra thick](6,1.5)--++(0,-.75) ++(1,0)--++(0,.75) ++(1,0)--++(0,.-.75);
		\draw[ultra thick](6,-.75)--++(0,-.75) ++(1,0)--++(0,.75) ++(1,0)--++(0,.-.75);

		\begin{scope}[yshift = -.75cm]
		\braidline{3}{2}{6}{4}{2}
		\end{scope}
		\draw[ultra thick](6,-4.5)--++(0,-.75) ++(1,0)--++(0,.75) ++(1,0)--++(0,.-.75);
		\draw[ultra thick](6,-6.75)--++(0,-.75) ++(1,0)--++(0,.75) ++(1,0)--++(0,.-.75);

		\draw[ultra thick](12,-1.5)--++(0,-3) ++(1,0)--++(0,3) ++(1,0)--++(0,.-3);

		\draw[->] (3,0)--(5,0);
		\draw[->] (3,-6)--(5,-6);
		\draw[->] (3,-1)--(3.8,-2.6) (4.2,-3.4)--(5,-5);
		\draw[->] (3,-5)--(5,-1);
		\draw[->] (9,-1)--(11,-2);
		\draw[->] (9,-5)--(11,-4);
	\end{tikzpicture}\end{center}
\caption[A simplified reduced complex of $D$.]{\figbld{A simplified presentation of $D^0$: the reduced resolution complex of $D$.}}
\label{fig-r3D1simple}
\end{figure}

In the diagram for $C(D)$, we can think of $C$ as the contraction of the two thick edges. 
This works because $C(D_{000})\cong C(D_{110})\,\,(\cong C(D_{011}))$, and $C(D_{010})$ comes from $D_{010}$, which is $D_{000}\sqcup O$. 

Now, we represent $C(\hat D)$ by the diagram in Figure~\ref{fig-r3hD}. 
Note, like with $C(D)$, we have $C(\hat D_{000})\cong C(\hat D_{110})\cong C(\hat D_{011})$, and $C(\hat D_{010})$ comes from $\hat D_{010}$, which is $\hat D_{000}\sqcup O$. 
Thus, just as with $C(D)$, we can contract the two thick edges in the diagram, giving us $\hat C$.
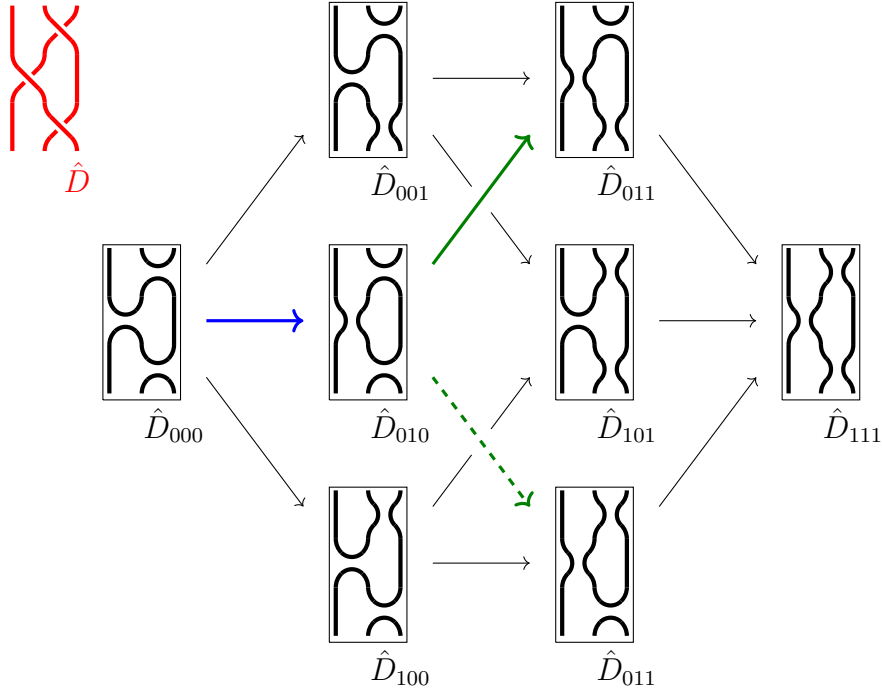
\begin{figure}
	\begin{center}\begin{tikzpicture}[scale=\rsc]
		\def\dx{7}
		\def\dy{-7.5}
		\long\def\diagbox#1#2#3{\draw (#1*\dx,#2*\dy) ++(-.2,-3.2) -- ++(2.4,0) -- ++(0,4.8) -- ++(-2.4,0) -- ++(0,-4.8);
		\draw (#1*\dx,#2*\dy) ++(2,-3.1) node[below]{$#3$};}
		\color{red}
		\threebraid{3}{-3}{0}{-2}{-1}{-2}
		\draw (-1,-3) node[below]{$\hat D$};
		\color{black}
		\threeresdiagtall{3}{0}{0}{-2}{-1}{-2}{1}
		\diagbox{0}{1}{\hat D_{000}}
		\diagbox{1}{0}{\hat D_{001}}
		\diagbox{1}{1}{\hat D_{010}}
		\diagbox{1}{2}{\hat D_{100}}
		\diagbox{2}{0}{\hat D_{011}}
		\diagbox{2}{1}{\hat D_{101}}
		\diagbox{2}{2}{\hat D_{011}}
		\diagbox{3}{1}{\hat D_{111}}

		\draw[->] (3,-6.5)--++(3,4);
		\draw[->,very thick, blue] (3, -8.25)--++(3,0);
		\draw[->] (3,-10)--++(3,-4);

		\draw[->] (10,-.75)--++(3,0); 
		\draw[->] (10,-2.5)--++(1.2,-1.6) ++(.6,-.8)--++(1.2,-1.6);
		\draw[->,very thick, green!50!black] (10,-6.5)--++(3,4);
		\draw[->,very thick, dashed, green!50!black] (10,-10)--++(3,-4);
		\draw[->] (10,-14)--++(1.2,1.6) ++(.6,.8)--++(1.2,1.6);
		\draw[->] (10,-15.75)--++(3,0);

		\draw[->] (17,-2.5)--++(3,-4);
		\draw[->] (17,-8.25)--++(3,0);
		\draw[->] (17,-14)--++(3,4);
		\end{tikzpicture}\end{center}
\caption[The resolution cube for the diagram after an RIII move.]{\figbld{The resolution cube for $\hat D$: the diagram after an RIII move.}}
\label{fig-r3hD}
\end{figure}

The contracted diagram is given in Figure~\ref{fig-r3hDC} and labeled according to how vertices will correspond with $C$. 
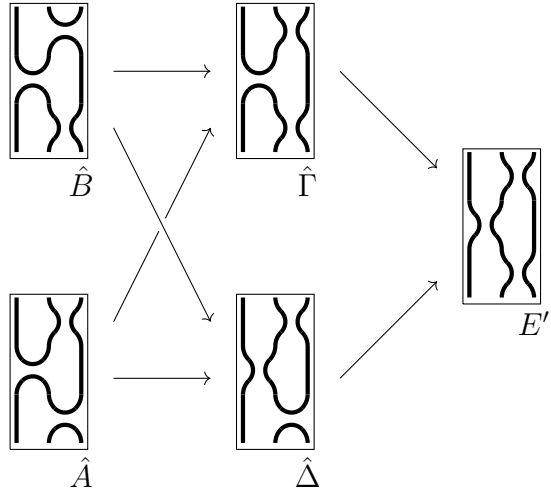
\begin{figure}
	\begin{center}\begin{tikzpicture}[scale=\rsc]
		\def\dx{7}
		\def\dy{-6}
		\long\def\diagbox#1#2#3{\draw (#1*\dx,#2*\dy) ++(-.2,-3.2) -- ++(2.4,0) -- ++(0,4.8) -- ++(-2.4,0) -- ++(0,-4.8);
		\draw (#1*\dx,#2*\dy) ++(2,-3.1) node[below]{$#3$};}
		\diagbox{1}{1/4}{\hat B}
		\diagbox{1}{7/4}{\hat A}
		\diagbox{2}{1/4}{\hat \Gamma}
		\diagbox{2}{7/4}{\hat \Delta}
		\diagbox{3}{1}{E'}
		\threeres{3}{7}{1}{-2}{-1}{-2}{0}{0}{1}
		\threeres{3}{7}{7}{-2}{-1}{-2}{1}{0}{0}
		\threeres{3}{14}{1}{-2}{-1}{-2}{1}{0}{1}
		\threeres{3}{14}{7}{-2}{-1}{-2}{1}{1}{0}
		\threeres{3}{21}{4}{-2}{-1}{-2}{1}{1}{1}
		\draw[->] (10,-2)--(13,-2); 
		\draw[->] (10,-9.75)--(11.4,-6.95) (11.6,-6.55)--(13,-3.75);
		\draw[->] (10,-3.75)--(13,-9.75);
		\draw[->] (10,-11.5)--(13,-11.5);
		\draw[->] (17,-2)--(20,-5);
		\draw[->] (17,-11.5)--(20,-8.5);
	\end{tikzpicture}\end{center}
\caption[The resolution complex of contracted $\hat D$.]{\figbld{The resolution complex of the contracted version of $\hat D$.}}
\label{fig-r3hDC}
\end{figure}

These relations give us a quasi-isomorphism $\hat\rho: C(\hat D)\too \hat C$ and as before, the quasi-isomorphism is the identity on $C(\hat D_{111})$. 
The simplified diagram for $C(\hat D^1)$ is presented in Figure~\ref{fig-r3hDCsimple}. 
We note that, except for swapping the position of the leftmost nodes, this corresponds exactly with the simplified diagram of $C$, thus, $C\cong\hat C$. 
\begin{figure}
	\begin{center}\begin{tikzpicture}[scale=\rsc]
		\def\dx{6}
		\def\dy{-6}
		\long\def\diagbox#1#2#3{\draw (#1*\dx,#2*\dy) ++(-.2,-1.7) -- ++(2.4,0) -- ++(0,3.4) -- ++(-2.4,0) -- ++(0,-3.4);
		\draw (#1*\dx,#2*\dy) ++(2,-1.6) node[below]{$#3$};}

		\diagbox{0}{0}{\hat B}
		\diagbox{0}{1}{\hat A}
		\diagbox{1}{0}{\hat \Gamma}
		\diagbox{1}{1}{\hat \Delta}
		\diagbox{2}{.5}{\hat E}
		\braidline{3}{2}{0}{0}{2}
		\braidline{3}{1}{0}{1}{2}

		\braidline{3}{1}{0}{4}{2}
		\braidline{3}{2}{0}{5}{2}

		\begin{scope}[yshift = -.75cm]
		\braidline{3}{1}{6}{0}{2}
		\end{scope}
		\draw[ultra thick](6,1.5)--++(0,-.75) ++(1,0)--++(0,.75) ++(1,0)--++(0,.-.75);
		\draw[ultra thick](6,-.75)--++(0,-.75) ++(1,0)--++(0,.75) ++(1,0)--++(0,.-.75);

		\begin{scope}[yshift = -.75cm]
		\braidline{3}{2}{6}{4}{2}
		\end{scope}
		\draw[ultra thick](6,-4.5)--++(0,-.75) ++(1,0)--++(0,.75) ++(1,0)--++(0,.-.75);
		\draw[ultra thick](6,-6.75)--++(0,-.75) ++(1,0)--++(0,.75) ++(1,0)--++(0,.-.75);

		\draw[ultra thick](12,-1.5)--++(0,-3) ++(1,0)--++(0,3) ++(1,0)--++(0,.-3);

		\draw[->] (3,0)--(5,0);
		\draw[->] (3,-6)--(5,-6);
		\draw[->] (3,-5)--(3.8,-3.4) (4.2,-2.6)--(5,-1);
		\draw[->] (3,-1)--(5,-5);
		\draw[->] (9,-1)--(11,-2);
		\draw[->] (9,-5)--(11,-4);
	\end{tikzpicture}\end{center}
\caption[A simplified presentation of contracted $\hat D$.]{\figbld{A simplified presentation of the resolution complex for the contracted version of $\hat D$.}}
\label{fig-r3hDCsimple}
\end{figure}
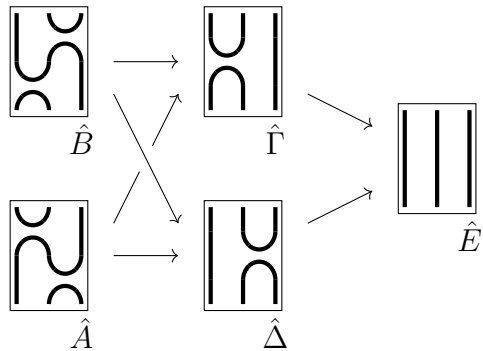
\end{proof}

Since the chain complexes are bounded, having proved Propositions~\ref{prop-r1}-\ref{prop-r3},  a more precise formulation of Theorem~\ref{thm-result} from Subsection~\ref{sec-result} follows.

\begin{thm}
\label{thm-inv}
Given two diagrams $D$ and $D'$ of the same transverse link $L$, there is an isomorphism $\rho:Kh'(D)\too Kh'(D')$ 
such that $\psi(D')=\rho(\psi(D))$. 
\end{thm}

Hence, we can unambiguously write $\psi(L)$ instead of $\psi(D)$. 

Since $Kh'(L)$ is not known to be natural, $Kh'(L)$ is only (currently) known to be well-defined up to automorphism. 
Above, we have shown that there is a well-defined map $\rho$ that takes $\psi(D)$ to $\pm\psi(D')$ associated to any sequence of transverse Markov moves from $D$ to $D'$. 
In particular, whether $\psi(D)$ vanishes, whether $\psi(D)$ is $n$-torsion, or  whether $\psi(D)$ is divisible by $n$ are all well-defined invariants of the transverse link type.

\section{Reduced Odd Plamenevskaya Invariant}
\label{sec-rh}
In this section, we examine the invariant $\psi(L)$ in reduced odd Khovanov homology. 
In Section~\ref{ssec-redcc}, we define the reduced homology $\overline C(D)$, first defined in~\cite[Section~4]{ORSz}. 
The relationship between this chain complex and the full odd Khovanov chain complex is stated in Proposition~\ref{prop-ccsplit}, and is extended to their homologies in Corollary~\ref{cor-okhsplit}.

In Section~\ref{ssec-ropi}, we define a reduced version of the odd Plamenevskaya invariant, a class in the reduced homology $\overline\psi(D)\in\overline{Kh'}(L)$.
With our maps defined explicitly, it will be possible to identify $\psi(D)$ in the full homology with the reduced version of the invariant under the relationship between the full and reduced homologies, which we state precisely in Proposition~\ref{prop-opiinrokh}. 
From this, it follows that $\overline{\psi}(D)$ is a transverse link invariant, stated as Corollary~\ref{cor-rpsitli}. Thus, we can unambiguously write $\overline\psi(L)$. 

\subsection{Reduced Odd Khovanov Homology}
\label{ssec-redcc}
Reduced odd Khovanov homology is defined first on the level of complexes associated to a link diagram $D$. 
There are two definitions for the reduced chain complex: a basepoint-dependent definition, and a basepoint-independent definition. 
They are isomorphic, and the basepoint-dependent definition is useful for proving properties of the reduced odd Khovanov homology. 
For the basepoint-dependent one, we take a point, $p\in D$ not at one of the crossings. 
In each resolution, this point will belong to a particular circle. 
Choose labelings of the circles so that in every resolution this circle is labelled $a_p$, and define
	\[\overline{C}^{(p)}(D)= v_p\wedge C(D)\subset C(D).\]
As a consequence of Proposition~\ref{prop-chipd}, we will see for $p,q\in D$, that $\overline{C}^{(p)}(D)\cong\overline{C}^{(q)}(D)$. 

The base-point-independent definition comes from the exterior algebras $\Lambda^* V_\alpha$ that make up the direct sum that defines $C(D)$. 
For each resolution $\alpha$ with $V_\alpha=\langle v_1,\dots,v_n\rangle$, we define $\varphi_\alpha:V_\alpha\too R$ by
	\[\varphi_\alpha: \sum r^iv_i\mapsto\sum r^i.\]
We define $\Lambda^*_\circ V_\alpha$ to be the subalgebra generated by the kernel of $\varphi_\alpha$. 
That is, $\Lambda^*_\circ V_\alpha = \Lambda^*(\ker\varphi_\alpha)$. 
Then, we define $\overline{C}(D)$ to be the subcomplex of $C(D)$ corresponding to sum of all $\Lambda^*_\circ V_\alpha$. 
That this is a subcomplex is a consequence of the following proposition.

\begin{prop}
\label{prop-diffonred}
For each $r$, $d(\overline C^r(D))\subset \overline C^{r+1}(D)$. That is, $\overline C(D)$ is a subcomplex of $C(D)$.
\end{prop}

The proof is straightforward.
Recall that an edge $e$ in the $\mathcal X$-cube of resolutions corresponds to a pair of resolutions, $\alpha_0,\alpha_1$, which differ at a single crossing $x$ such that $\alpha_0(x)=0$ and $\alpha_1(x)=1$. In the cube of $R$-modules, we have a map 
	\[F_e:\Lambda^*V_{\alpha_0}\too\Lambda^*V_{\alpha_1}.\]
To prove Proposition~\ref{prop-diffonred}, we must show for any such edge $e$, 
	\[F_e(\Lambda^*_\circ V_{\alpha_0})\subset\Lambda^*_\circ V_{\alpha_1}.\]
This is straightforward. 
Proposition~\ref{prop-diffonred} follows immediately. 

\begin{prop}
\label{prop-ccsplit} 
There is an isomorphism
\[C(D)\cong\overline C(D)\oplus\overline C(D).\]
\end{prop}	 

It is straightforward to show that for resolutions $\alpha_0$ and $\alpha_1$ connected by an edge, with corresponding map $F$ between them that
	\[F(\Lambda^*_\circ V_{\alpha_0})\subset\Lambda^*_\circ V_{\alpha_1}.\]
The proof breaks into two cases: one in which the map corresponds to a merge cobordism, and the other in which the map corresponds to a split. Hence, it follows that $\overline C(D)$ is a chain subcomplex of $C(D)$.

\begin{prop}[Ozsv\'ath-Rasmussen-Szab\'o, {\cite[Lemma 4.1]{ORSz}}]
\label{prop-chipd}
For any $p\in D$, there is an isomorphism 
	\[X_D^p: \overline C^{(p)}(D)\too \overline C(D).\]
\end{prop}

\begin{defn}
\label{defn-rokh}
The reduced odd Khovanov homology $\overline{Kh'}(L)$ is defined to be the homology of $\overline C(D)$. 
\end{defn}

\begin{cor}[Ozsv\'ath-Rasmussen-Szab\'o, {\cite[Proposition 1.7]{ORSz}}]
\label{cor-okhsplit}
\[Kh'_{m,s}(L) = \overline{Kh'}_{m,s-1}(L)\oplus\overline{Kh'}_{m,s+1}(L).\]
\end{cor}

In Section~\ref{ssec-ropi}, we define $\overline\psi(D)$ and prove it to be invariant, and we identify it with $\psi(L)$.

\subsection{The Reduced Odd Plamenevskaya Invariant}
\label{ssec-ropi}
\begin{defn}
\label{defn-rtpsi}
As in definition~\ref{defn-opi}, 
we let $\alpha'$ be the resolution in which the braid representation is separated into $b$ parallel bands. 
Again, this is the resolution in which there is a 0-smoothing for every positive crossing and a 1-smoothing for every negative crossing. 
The reduced odd Plamanevskaya invariant $\overline{\psi}(D)$ is a generator of $\Lambda^{b-1}_\circ V_{\alpha'}$. 
\end{defn}

As with the unreduced invariant, we will abuse notation to refer to both the element in the chain complex and its class in homology (shown to be well-defined in Proposition~\ref{prop-tpsicycle}) by $\overline\psi(D)$. 

\begin{prop}
\label{prop-tpsicycle}
$\overline{\psi}(D)$ is a cycle. 
\end{prop}

\begin{proof}
As in showing the unreduced invariant was a cycle, the differential out of $\Lambda^*_\circ V_{\alpha'}$ is 0 or a sum of merges. 
Without loss of generality, we take a labeling of $\alpha'$ so that the merged circles are $a_1$ and $a_2$, and we use 
	\[\{v_{i-1}-v_i\,|\,2\leq i\leq b\}\]
as the basis for $\ker\varphi_{\alpha'}$, generating a basis for $\Lambda^*_\circ V_{\alpha'}$.  Then
	\[{\overline\psi}(D) = \bigwedge_{i=2}^b (v_{i-1}-v_i).\]
Thus
	\begin{align*}
	F_M({\overline\psi}(D) ) 
		&=F_M\left((v_1-v_2)\wedge\bigwedge_{i=3}^b(v_{i-1}-v_i)\right)\\
		&=0\wedge\bigwedge_{i=3}^b(v_{i-1}-v_i)\\
		&=0.\qedhere
	\end{align*}
\end{proof}

Since ${\overline{\psi}}(D)$ is a cycle, it induces an element in the reduced homology $\overline{Kh'}(L)$.
To identify the invariant with the reduced version, we will follow it in the chain complexes explicitly though the identification of $C(D)$ and $\overline C(D)\oplus\overline C(D)$. 
We do this below in a series of lemmas. 
Our proof will use both $\overline C(D)$ and $\overline C(D\sqcup O)$ where $O$ is an additional unknot labeled $a_0$. 
As a braid, the diagram for $D\sqcup O$ is the diagram for $D$ with an additional strand that is not connected by any crossings. 
For the vector spaces from which the exterior algebras forming $C(D)$ are built, we will use the notation
	\[V_\alpha = \langle v_1,\dots,v_n\rangle.\]
Each resolution in $C(D)$ corresponds to a resolution in $C(D\sqcup O)$, which will be built out of vector spaces 
	\[V_\alpha' = \langle v_0,v_1\dots,v_n\rangle.\]
 We will often use the fact that if $V' = V \oplus\langle v_0\rangle$, then
	\[\Lambda^* V' = \langle\omega,v_0\wedge\omega\,|\,\omega\in\Lambda^*V\rangle.\]

In Proposition~\ref{prop-opiinrokh}, we focus on the invariant defined in the exterior algebra constructed from the resolution $\alpha'$. 

\begin{lem}
\label{lem-Phi1}
There is an isomorphism, $\Phi_1$ between $C(D)$ and $\overline C^{(p)}(D\sqcup O)$, where we take $p\in O$.
\end{lem}

\begin{proof}
We take a sign assignment on the resolution cube for $C(D)$ and choose the same sign assignment on $C(D\sqcup O)$. This is possible because the squares of each have the same commutativity types. This induces the differential on $\overline C^{(p)}(D\sqcup O)$. 
 On the level of the exterior algebras corresponding to each resolution, this comes from a map
	\[\phi_1:\Lambda^*V_\alpha\too v_0\wedge \Lambda^*V_\alpha'.\]
Since $\Lambda^*V_\alpha' = \langle\omega,v_0\wedge\omega\,|\,\omega\in\Lambda^*V_\alpha\rangle$, then
	\[v_0\wedge\Lambda^*V_\alpha'=\langle v_0\wedge\omega, v_0\wedge\omega\,|\,\omega\in\Lambda^*V_\alpha\rangle= \langle v_0\wedge\omega\,|\,\omega\in\Lambda^*V_\alpha\rangle.\]
Thus, $\phi_1$, for $\omega\in\Lambda^*V_\alpha$, is given by
	\[\omega\mapsto v_0\wedge\omega.\]
This extends to all resolutions as a map $\Phi_1$. 

Since $v_0$ is unmodified by the differentials as it corresponds to a circle $O$, which is unaffected by any merges or splits in the corresponding cobordisms, then
	\[d(v_0\wedge\omega)=v_0\wedge d\omega.\]
Thus, $\Phi_1$ is a chain map. 
\end{proof}

\begin{lem}
\label{lem-Phi2}
 There is an isomorphism $\Phi_2$ between $\overline C^{(p)}(D\sqcup O)$ and $\overline C^{(q)}(D\sqcup O)$ where $p$ is a point in $O$ and $q$ is a point in $D$.
\end{lem}

\begin{proof}
This is a consequence of Proposition~\ref{prop-chipd}. 
We have
	\[\Phi_2 = (X^q_{D\sqcup O})^{-1}\circ X^p_{D\sqcup O}.\qedhere\]
\end{proof}

\begin{lem}
\label{lem-Phi3}
There is an isomorphism, $\Phi_3$ between $\overline C^{(q)}(D\sqcup O)$ and $\overline C^{(q)}(D)\oplus\overline C^{(q)}(D)$.
\end{lem}

\begin{proof}
From the exterior algebras, we define
	\[\phi_3:v_1\wedge\Lambda^*V_\alpha'\too (v_1\wedge\Lambda^*V_\alpha)\oplus (v_1\wedge\Lambda^*V_\alpha),\]
by observing
	\[v_1\wedge\Lambda^*V_\alpha' 
		= \langle v_1\wedge\eta,v_0\wedge v_1\wedge\eta\,|\,\eta\in\Lambda^*\hat V\rangle 
		= \langle \omega,v_0\wedge\omega\,|\,\omega\in v_1\wedge \Lambda^*V_\alpha\rangle,\]
so the map is given by
	\begin{align*}
		\omega\mapsto (\omega,0)
			&\qquad \omega\in v_1\wedge\Lambda^*V_\alpha\subset v_1\wedge\Lambda^*V_\alpha'\\
		v_0\wedge\omega\mapsto(0,\omega)
			&\qquad\omega\in v_1\wedge\Lambda^*V_\alpha\subset v_1\wedge\Lambda^*V_\alpha'.
	\end{align*}
It is clear that $\Phi_3$ is a chain map.
\end{proof}

\begin{lem}
\label{lem-Phi4}
There is an isomorphism $\Phi_4$ between $\overline C^{(q)}(D)\oplus\overline C^{(q)}(D)$ and $\overline C(D)\oplus \overline C(D)$.
\end{lem}	

\begin{proof}
This is a consequence of Proposition~\ref{prop-chipd}. 
We have
	\[\Phi_4 := X^q_D\oplus X^q_D.\qedhere\]
\end{proof}

\begin{cor}
\label{cor-ccexplicit}
There is an isomorphism, $\Phi: C(D)\too\overline C(D)\oplus\overline C(D)$, given by
	\[\Phi := \Phi_4\circ\Phi_3\circ\Phi_2\circ\Phi_1.\]
\end{cor}

Note, on the level of the exterior algebras, this corresponds to a map
	\[\phi:\Lambda^* V_\alpha\too \Lambda^*_\circ V_\alpha\oplus\Lambda^*_\circ V_\alpha\]
given by
	\[\phi = (\phi_4\oplus\phi_4)\circ \phi_3\circ\phi_2\circ\phi_1.\]

\begin{prop} 
\label{prop-opiinrokh}
The inclusion map $\overline{Kh'}(L)\lhook\joinrel\longrightarrow Kh'(L)$ sends the reduced odd Plamenevskaya invariant $\overline{\psi}(D)\in \overline{Kh'}_{0,sl(L)+1}(L)$ to the odd Plame\-nev\-skaya invariant $\pm\psi(D)\in Kh'_{0,sl(L)}$. 
\end{prop}

\begin{proof}
	Let $\alpha'$ be the resolution corresponding to the invariant, and let $V_{\alpha'}'$ be the vector space corresponding to the same resolution in $D\sqcup O$, labeled as before. 
As given in Proposition~\ref{prop-chipd}, 
we define 
	\[\phi_2 =(\chi^q_{D\sqcup O})^{-1}\circ\chi^p_{D\sqcup O}:
		v_0\wedge\Lambda^*V_{\alpha'}'\too v_i\wedge\Lambda^*V_{\alpha'}',\]
and 
	\[\phi_4=\chi^q_D:v_0\wedge\Lambda^*V_{\alpha'}:\Lambda^*_\circ V_{\alpha'}.\]
Then,
	\[\phi(\tilde\psi(D)) =  (\phi_4\oplus\phi_4)\circ \phi_3\circ\phi_2\circ\phi_1(\tilde\psi(D)).\]
We note that since $\phi_2$ is an isomorphism of $R$-modules that is degree 0 with respect to the natural grading of the exterior algebra, and $v_0\wedge\tilde\psi(D)$ has top degree in the exterior algebra, it follows that 
	\[\phi_2(v_0\wedge\tilde\psi(D))=\pm v_0\wedge\tilde\psi(D).\]
Note, $\phi_4=\chi^q_D$ is an isomorphism of $R$-modules that is degree $-1$ with respect to the exterior algebra. 
Furthermore, $\tilde\psi(D)\in v_i\wedge\Lambda^*V_{\alpha'}$ has top degree in the exterior algebra, $\deg_k \tilde\psi(D) = b$, and $\tilde{\overline\psi}(D)$ has top degree in $\Lambda^*_\circ V_{\alpha'}$, $\deg_k \tilde{\overline\psi}(D) = b-1$. 
Thus, it follows that
	\[\phi_4(\tilde\psi(D))=\pm\tilde{\overline\psi}(D).\]

Thus, we have
	\[\begin{array}{rccc}
		\phi_1: &\tilde\phi(D)&\mapsto&v_0\wedge\tilde\psi(D)\\
		\phi_2:&v_0\wedge\tilde\psi(D)&\mapsto&\pm v_0\wedge\tilde\psi(D)\\
		\phi_3:&\pm v_0\wedge\tilde\psi(D)&\mapsto&(0,\tilde\psi(D))\\
		\phi_4:&\tilde\psi(D)&\mapsto&\tilde{\overline\psi}(D).
	\end{array}\]
	
Therefore,
		\begin{align*}
		\phi(\tilde\psi(D)) 
			&=(\phi_4\oplus\phi_4)\circ \phi_3\circ\phi_2\circ\phi_1(\tilde\psi(D))\\
			&=(\phi_4\oplus\phi_4)(0,\tilde\psi(D))\\
			&=(0,\pm\tilde{\overline\psi}(D)),
		\end{align*}
and thus on the chain complex, we have 
$\Phi(\tilde\psi(D))=(0,\pm\tilde{\overline{\psi}}(D))$.

Note that $\tilde\psi(D)\in\Lambda^b V_{\alpha'}$ and $\Phi(\tilde\psi(D))\in\Lambda^{b-1}(\ker\varphi_{\alpha'})$. 
Thus
	\begin{align*}
	\deg_Q \Phi(\tilde\psi(D)) 
		&= (\dim \ker\varphi_{\alpha'}) - 2(b-1) + n_+ - 2n_- + |\alpha'|\\
		&= (\dim V_{\alpha'} - 1) - 2(b-1) + n_+ - 2n_- + |\alpha'|\\
		&=(\dim V_{\alpha'}) - 2b + n_+ - 2n_- + |\alpha'| + 1\\
		&=\deg_Q \tilde\psi(D) + 1\\
		&=sl(L) + 1.
	\end{align*}
Therefore, $\overline \psi(D)\in\overline{Kh'}_{0,sl(L)+1}(D)$.
\end{proof}

\begin{cor}
\label{cor-rpsitli}
$\overline{\psi}(D)$ is a transverse link invariant. More precisely, if $D$ and $D'$ are diagrams for a transverse link $L$, there is a quasi-isomorphism
	\[\overline\rho:\overline C(D)\too \overline C(D')\]
such that
	\[\overline\rho(\overline\psi(D)) = \pm\overline\psi(D').\]
\end{cor}

\begin{proof}
Let $D$ and $D'$ both be diagrams for the same transverse link, $L$. Let
	\[\Phi:C(D)\too \overline{C}(D)\oplus\overline{C}(D)\]
and
	\[\Phi':C(D')\too \overline{C}(D')\oplus\overline{C}(D')\]
be the isomorphisms defined in Corollary~\ref{cor-ccexplicit} for $D$ and $D'$ respectively. 
Let $\rho:C(D)\too C(D')$ be the quasi-isomorphism on the chain complexes given by the invariance of odd Khovanov homology.
We have a quasi-isomorphism,
	\[\Phi'\circ\rho\circ\Phi^{-1}:\overline{C}(D)\oplus\overline{C}(D) \too C(D)\too C(D')\too \overline{C}(D')\oplus\overline{C}(D')\]
with
	\[\Phi'\circ\rho\circ\Phi^{-1}(0,\overline\psi(D))=(0,\pm\overline\psi(D')).\qedhere\]
\end{proof}

\begin{cor}\label{cor-altredzero}For transverse, alternating knot, $K$, if $sl(K)+1\neq \sigma(K)$, then $\psi(K)=0$. \end{cor}
\begin{proof} By the previous proposition, $\psi(K)\in\overline{Kh'}_{0,sl(K)+1}(K)$. Since $K$ is alternating, by~\cite[Proposition~5.2]{ORSz}, then $\overline{Kh'}_{m,s}(K)=0$ whenever $s-2m\neq\sigma(K)$, and thus, since we assume $sl(K)+1\neq\sigma(K)$, it follows that
	\[\overline{Kh'}_{0,sl(K)+1}(K)=0.\]
Thus, $\psi(K)=0$.\end{proof}

\section{Properties}
\label{sec-prop}

\begin{prop}
\label{prop-unknot}
For the standard transverse unknot, $O$, up to a sign, $\psi(O)$ is a generator of $Kh'_{0,-1}(O)\cong R$. 
\end{prop}

\begin{proof}
We take the trivial diagram of $O$. 
Thus, there are no crossings, and there is only one resolution corresponding to the unique function in $\{0,1\}^\emptyset$. 
The one circle in this resolution is exactly our presentation of $O$, and it has generator $v_0$. 
Thus, we have chain complex
	\[0\too \Lambda^*\langle v_0\rangle\too 0,\]
and $Kh'_0=R[-1]\oplus R[1]$.
Since $\psi(O)=v_0\in\Lambda^1\langle v_0\rangle=\langle v_0\rangle= R\cdot v_0$ and its grading is $(0,-1)$, the proof is complete.
\end{proof}

Although trivial, the previous proposition is important to note because $\psi$ is defined first on the level of being a cycle $\tilde\psi$ in the chain complex. 
Thus, it is not immediate that there are links in which $\tilde\psi$ is not also a boundary. 
The following proposition, which is analogous to \cite[Propostion~3]{Pla} however, gives us one condition in which we can guarantee that $\psi$ is a boundary.

\begin{prop}
\label{prop-negstab}
If $L$ is the negative stabilization of another transverse link, then $\psi(L)=0$. 
\end{prop}

\begin{proof}
Assume $L$ is the negative stabilization of another transverse link. 
Thus for an appropriate choice of braid representation, there is a part of the diagram of $L$ that contains just the negative stabilization with a single negative crossing, and in the chain complex there are two cubes of resolutions corresponding to the 0- and 1-smoothings at this crossing. 
We will call the diagram $D$, shown in Figure~\ref{fig-negstab}.
We note that $d^\star$ here is given by the split cobordism. 
If we label the circles of $D_0$ and $D_1$ that extend beyond the diagram $a_1$, with corresponding generator $v_1$, and the circle entirely contained in the diagram of $D_1$ by $a_0$, with corresponding generator $v_0$, then the map $d^\star:D_0\too D_1$ is given by $\omega\mapsto (v_0-v_1)\wedge\omega$.
Note that $\tilde\psi(L)\in C(D_1)$. 
 \begin{figure}
	\begin{center}\begin{tikzpicture}[scale=0.75]
		\color{red}
		\braidline{3}{1}{1}{0}{4}
		\braidline{3}{2}{1}{1}{-1}
		\braidline{3}{1}{1}{2}{5}

		\draw (2,-3.1) node[below]{$D$};

		\color{black}
		\braidline{3}{1}{8}{0}{4}
		\braidline{2}{2}{8}{1}{2}
		\braidline{3}{1}{8}{2}{5}

		\braidline{3}{1}{14}{0}{4}
		\braidline{2}{2}{14}{1}{3}
		\braidline{3}{1}{14}{2}{5}
		\def\dx{8}
		\def\dy{-3}
		\long\def\diagbox#1#2#3{\draw (#1*\dx,#2*\dy) ++(-.2,-3.2) -- ++(2.4,0) -- ++(0,4.9) -- ++(-2.4,0) -- ++(0,-4.9);
		\draw (#1*\dx,#2*\dy) ++(2,-3.1) node[below]{$#3$};}

		\diagbox{1}{0}{D_0}
		\diagbox{1.75}{0}{D_1}

		\draw[->] (10.5,-.75) -- (13.5,-.75);
	\end{tikzpicture}\end{center}
 \caption[A diagram of a negative stabilzation.]{\figbld{A diagram of $D$ focused on an added negative stabilzation of $L$, and its resolution cube.}}
 \label{fig-negstab}
\end{figure}
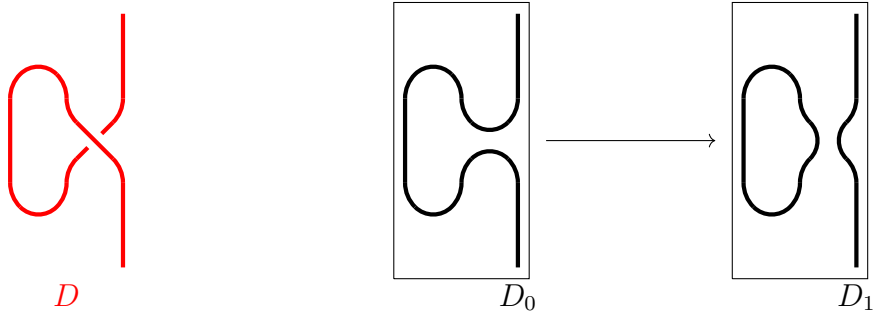

It suffices to show that $\tilde\psi$ is a boundary, so we construct $\tilde\phi\in C(D_0)$ such that $d\tilde\phi=\tilde\psi$. 
We note that in the cube of resolutions $D_1$, there is a specific resolution in which $\tilde\psi$ resides. We consider the corresponding resolution in $D_0$. 
That is, the resolution that has a 0-smoothing at every positive crossing and a 1-smoothing at every negative crossing except the one added by the negative stabilization. 
If there are $n$ generators $v_1,\dots,v_n$ in the vector space corresponding to this resolution, then we let $\tilde\phi=v_1\wedge\dots\wedge v_n$. 
Note, there is a natural isomorphism between $C(D_0)$, and the chain complex associated to the link $L'$ to which a negative stabilization was added. 
Up to a sign, our constructed element is the image of $\psi(L')$ under this isomorphism. 

Computing $d\tilde\phi$, we note there are two sources where the resolution for $L$ has 0-smoothings: the positive crossings of $L$ and the negative crossing introduced by the negative stabilization. 
On the positive crossings, given our braid representation, a change from a 0-smoothing to a 1-smoothing corresponds to a merge cobordism. 
Thus, on each summand in the differential $\tilde\phi\mapsto0$, as two factors in the wedge product would be identified. 
Hence, 
	\[d\tilde\phi=d^\star\tilde\phi=(v_0-v_1)\wedge\tilde\phi =v_0\wedge v_1\wedge\dots\wedge v_n = \tilde\psi.\qedhere\]
\end{proof}

The following proposition is an analog of \cite[Theorem~4]{Pla}. However, as it is not yet know if the odd Khovanov homology is functorial, the proposition below is necessarily weaker. 

\begin{prop}
\label{prop-posx}
Suppose we have a transverse link $L$ with diagram $D$, and $L_0$ is the transverse link with diagram $D_0$ obtained by replacing a positive crossing in $L$ with the 0-smoothing. 
There is a homomorphism,
	\[p: Kh'(L)\too Kh'(L_0)\]
such that $p(\psi(L))=\pm\psi(L_0)$. 
\end{prop}

\begin{proof}
In Figure~\ref{fig-posx}, we have the composition of the cobordism from attaching a 1-handle on one side of the positive crossing, and the (R1) move to undo the twist. 
We examine the diagram of the cobordism in the vertex of the resolution cube in which $\psi(D)$ resides (Figure~\ref{fig-posx}: top), and the corresponding resolutions of $D_0$ (Figure~\ref{fig-posx}: bottom) both with and without the extra twist by (R1). 

\begin{figure}
	\begin{center}\begin{tikzpicture}[scale=0.75]
		\draw [ultra thick] (-5,-3) -- ++(0,1.141) arc (180:90:.604cm) -- ++(0.146,0) arc (270:315:.604cm) -- ++(.646,.646) arc(315:360:.604cm)--++(0,1.5) ;
		\draw [ultra thick] (-5,1.5)--++(0,-1.141) arc(180:270:.604cm)--++(0.146,0) arc(90:45:.604cm) -- ++(.2,-.2) ++ (.246,-.246) -- ++(.2,-.2) arc(45:0:.605cm)--++(0,-1.5);
		\draw [very thick, densely dotted] (-4.2,-1.255) -- (-4.2,-.225);
		\draw [very thick, densely dotted] (-4.7,-1.255) -- (-4.7,-.245);

		\draw (-4,-3) node[below] {$D$};

		\draw [thick,->] (-2.5,-.75) -- ++(3,0);
		\draw (-2.5,-.75) ++(1.5,0) node[above] {1-handle};
		\draw [ultra thick] (1,1.5) -- ++(0,-1.141) arc(180:225:.604cm) arc (45:0:.604cm)-- ++(0,-.504);
		\draw [ultra thick] (1,-3) -- ++(0,1.141) arc(180:135:.604cm) arc (-45:0:.604cm);
		\draw [ultra thick] (3,1.5) -- ++(0,-1.5) arc(360:315:.604cm) -- ++(-.546,-.546) arc(315:45:.32cm) -- ++(.1,-.1) ++(.246,-.246) -- ++(.2,-.2) arc(45:0:.604cm)--++(0,-1.5);

 		\draw [thick,->] (3.5,-.75) -- ++(3,0);
		\draw (3.5,-.75) ++(1.5,0) node[above] {(R1)};

		\draw (2,-3) node[below] {$D_0'$};

		\braidline{1}{3}{7}{0}{1}
		\braidline{1}{1}{7}{1}{7}
		\braidline{1}{3}{7}{2}{1}

		\braidline{1}{3}{9}{0}{1}
		\braidline{2}{1}{8}{1}{6}
		\braidline{1}{3}{9}{2}{1}

		\draw (8,-3) node[below] {$D_0$};
	\end{tikzpicture}\end{center}

	\begin{center}\begin{tikzpicture}[scale=0.75]
		\draw [ultra thick] (-5,-3) -- ++(0,1.141) arc (180:90:.604cm) -- ++(0.196,0) arc (270:315:.604cm)  
 arc(-45:45:.457cm) arc(45:90:.604cm) -- ++(-.196,0) arc(270:180:.604cm) -- ++(0,1.141);
 		\draw (-5,-3) node[below] {$v_1$};
		\draw (-3,-3) node[below] {$v_2$};
 
		\braidline{1}{3}{-3}{0}{1}
		\braidline{2}{1}{-4}{1}{6}
		\braidline{1}{3}{-3}{2}{1}

		\draw [thick,->] (-2.5,-.75) -- ++(3,0);
		\draw (-2.5,-.75) ++(1.5,0) node[above] {$F_S$};
		\draw (-2.5,-.75) ++(1.5,0) node[below] {$(v_1-v_0)\wedge$};

		\draw (1,-3) node[below] {$v_1$};
		\draw (2,-1.1) node[below] {$v_0$};
		\draw (3,-3) node[below] {$v_2$};

		\braidline{1}{3}{1}{0}{1}
		\braidline{1}{1}{1}{1}{7}
		\braidline{1}{3}{1}{2}{1}
 
		\draw [ultra thick] (2,-.75) circle (.32cm);
 
		\braidline{1}{3}{3}{0}{1}
		\braidline{2}{1}{2}{1}{6}
		\braidline{1}{3}{3}{2}{1}

		\draw [thick,->] (3.5,-.75) -- ++(3,0);
		\draw (3.5,-.75) ++(1.5,0) node[above] {$F_M$};

		\braidline{1}{3}{7}{0}{1}
		\braidline{1}{1}{7}{1}{7}
		\braidline{1}{3}{7}{2}{1}
		
		\braidline{1}{3}{9}{0}{1}
		\braidline{2}{1}{8}{1}{6}
		\braidline{1}{3}{9}{2}{1}
 
 		\draw (7,-3) node[below] {$v_1$};
		\draw (9,-3) node[below] {$v_2$};
	\end{tikzpicture}\end{center}
\caption[The cobordism to remove a positive crossing.]{\figbld{Diagram in vertex for $\psi(D)$.} Top: the cobordism of the 0-smoothing to remove a positive crossing. Bottom: the corresponding resolution complex.}
\label{fig-posx}
\end{figure}
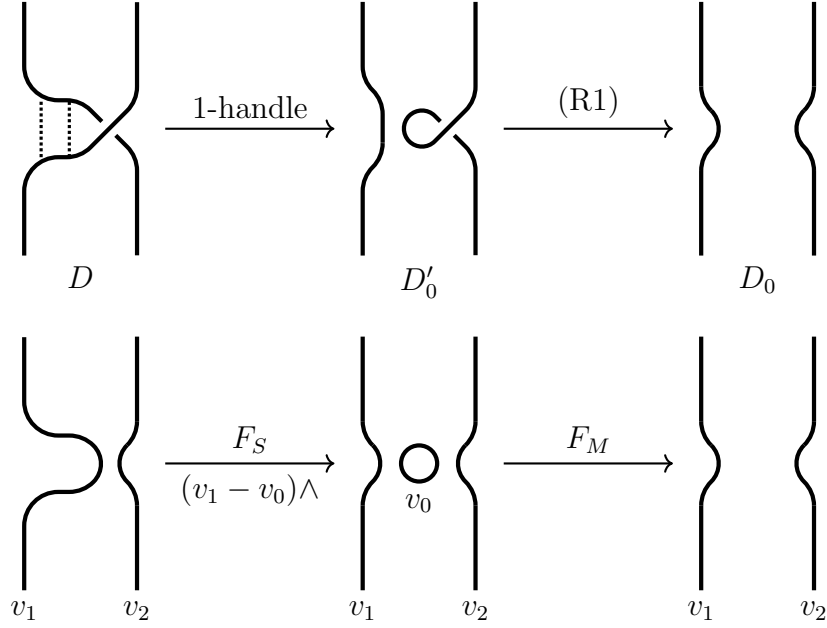

Since the first part of the composition comes from the split cobordism, we have
	\begin{align*}
	F_S(\psi(D)) 
		&=F_S(v_1\wedge v_2\wedge\dots)\\
		&=(v_1-v_0)\wedge (v_1\wedge v_2\wedge\dots)\\
		&=-v_0\wedge v_1\wedge v_2\wedge\dots\\
		&=-\psi(D_0')\\
	\intertext{and, as we have already seen in Proposition~\ref{prop-r1}}
		&=\pm\psi(D_0).\qedhere
	\end{align*}
\end{proof}

\begin{cor}
If $L$ can be represented by a quasi-positive braid, then $\psi(L)\neq0$. 
\end{cor}

\section{Computations}
\label{sec-comp}
In the original definition of odd Khovanov homology in \cite{ORSz}, the sign-assignment for the resolution cube is proved to exist but is not constructed. 
In his software for computing the even, (reduced) odd, and universal Khovanov homologies, Shumakovich implicitly provided an inductive construction of the sign assignment~\cite{Shu3}. 
An equivalent definition of the (unreduced) odd Khovanov is rigorously defined and proved to be equivalent to the original definition in Section~\ref{ssec-sa}. 
In Section~\ref{ssec-prog}, the author's program for computing the even and odd Plamenevskay invariants is summarized. 
We discuss observations made with the author's software in Section~\ref{ssec-compob}.

\subsection{The Sign Assignment}
\label{ssec-sa}
Let $\mathcal C$ be a category, and let $\mathcal I$ be a finite set. 
An \emph{$\mathcal I$-cube in $\mathcal C$} 
 consists of a collection of objects and morphisms of $\mathcal C$ as follows. 
For an $\mathcal I$-cube $C_*$, the \emph{vertices} of the cube are objects corresponding to elements of $\alpha\in\{0,1\}^{\mathcal I}$.
The \emph{height} of a vertex $C_\alpha$ is defined 
	\[|\alpha| = \sum_{i\in\mathcal I}\alpha(i).\]
The \emph{edges} are morphisms $e_\star:C_{\alpha_0}\too C_{\alpha_1}$ where $\alpha_0,\alpha_1\in\{0,1\}^{\mathcal I}$ differ only at a single index $i'\in\mathcal I$ with $\alpha_0(i')=0$ and $\alpha_1(i')=1$. 
The \emph{faces} correspond to quadruples $a_{00},a_{01},a_{10},a_{11}\in\{0,1\}^{\mathcal I}$ that agree except on a pair of indices, $i_1,i_2\in\mathcal I$ where $a_{jk}(i_1)=j$ and $a_{jk}(i_2)=k$ for $i,j=0,1$. 
	
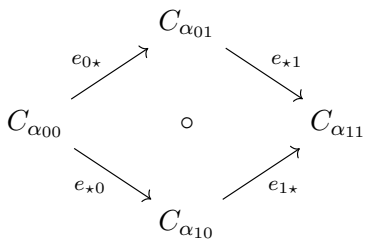
\begin{center}\begin{tikzcd}
		& C_{\alpha_{01}} \arrow[rd,"e_{\star1}"] & \\
		C_{\alpha_{00}} \arrow[ru,"e_{0\star}"] \arrow[rd,"e_{\star0}"'] &\circ& C_{\alpha_{11}}\\
		& C_{\alpha_{10}} \arrow[ru,"e_{1\star}"']
	\end{tikzcd}\end{center}
A \emph{commutative $\mathcal I$-cube in $\mathcal C$} is a $\mathcal I$-cube where all faces are required to commute. 
Thus, with edges labeled as above, faces satisfy $e_{\star1}\circ e_{0\star}=e_{1\star}\circ e_{\star0}$.
A \emph{skew $\mathcal I$-cube in $\mathcal C$} where $\mathcal C$ is an abelian category is defined similarly, with the condition on faces changed so that all faces are required to anti-commute.

For a (skew) cube $C_*$, its set of edges is denoted $\mathcal E(C_*)$ and for any $\alpha\in\{0,1\}^{\mathcal I}$ the subset of edges with $C_\alpha$ as its source is denoted $\mathcal E(\alpha)$. 
Similarly, the set of faces (or squares) is denoted $\mathcal S(C_*)$. 
Note our convention that $\star$ denotes a change from 0 to 1 in the indices of its vertex labels. 

A homomorphism between (skew) $\mathcal I$-cubes in $\mathcal C$ $f_*:C_*\too C'_*$, which is called a \emph{(skew) cube map}, consists of
a collection of morphisms $f_\alpha:C_\alpha\too C_\alpha'$, such that for each pair of corresponding edges $e[C_*]:C_{\alpha_0}\too C_{\alpha_1}$ and $e[C'_*]:C'_{\alpha_0}\too C'_{\alpha_1}$, we have
	\[f_{\alpha_1}\circ e_\star[C_*] = e_\star[C'_*]\circ f_{\alpha_0}.\]
$\mathcal I$-cubes in $\mathcal C$ with cube maps form the category of $\mathcal I$-cubes $\operatorname{Cube}_{\mathcal I}(\mathcal C)$. 
Commutative $\mathcal I$-cubes in $\mathcal C$ with cube maps for the category of $\mathcal I$-cubes $\operatorname{CCube}_{\mathcal I}(\mathcal C)$.
Skew $\mathcal I$-cubes in $\mathcal C$ with cube maps form the category of skew $\mathcal I$-cubes $\operatorname{SCube}_{\mathcal I}(\mathcal C)$.
Note,
	\[\operatorname{Cube}_{\emptyset}(\mathcal C) \cong\operatorname{SCube}_\emptyset(\mathcal C)\cong \mathcal C.\]
For any abelian category $\mathcal C$, $\operatorname{SCube}_{\mathcal I}(\mathcal C)$ is itself an abelian category.
For $C_*,D_*\in\operatorname{SCube}_{\mathcal I}(\mathcal C)$, we define
	\[E_* = C_*\oplus D_*\]
such that for each $\alpha\in\{0,1\}^{\mathcal I}$, we define
	\[E_\alpha = C_\alpha\oplus D_\alpha,\]
and for each edge $e:\alpha_0\too\alpha_1$, we define
	\[e[E_*] = e[C_*]\oplus e[D_*].\]

Given a skew cube map $f_*:C_*\too C'_*$, the mapping cone $\cone(f_*)$ is a skew $\mathcal J$-cube where $\mathcal J=\mathcal I\sqcup\{\hat\imath\}$.
Vertices $\alpha\in\{0,1\}^{\mathcal J}$ with $\alpha(\hat\imath)=0$ all correspond to the vertices of $C_*$, 
	\[\cone(f_*)_\alpha = C_\alpha\]
with edges between such vertices also coming from $C_*$. 
Similarly, vertices with $\alpha(\hat\imath)=1$ all correspond to the vertices of $C'_*$,
	\[\cone(f_*)_\alpha = C'_\alpha,\]
however the edges between such vertices are the negatives of the edges in $C'_*$.
For edge $e_\star:\alpha_0\too\alpha_1$ with $\alpha_0$ and $\alpha_1$ differing only at $\hat\imath$, for $\alpha = \alpha_0|_{\mathcal I}=\alpha_1|_{\mathcal I}$,
we have
	\[e_\star := f_\alpha.\]

For each $m\in\Z$, there is a functor from skew cubes to (co)chain complexes,
	\[\operatorname{Ch}^m:\operatorname{SCube}(\mathcal C)\too\operatorname{Ch}(\mathcal C),\]
where for an object $C_*$ in $\operatorname{SCube}(\mathcal C)$, 
	\[\operatorname{Ch}^m(C_*)_r := \bigoplus_{|\alpha|+m = r} C_\alpha,\]
and 
	\[d^r:= \bigoplus_{|\alpha|+m=r,e\in\mathcal E(\alpha)} e.\]
For a cube map $f_*:C_*\too C'_*$, $\operatorname{Ch}^m(f_*)$ is the chain map $g_*$ where
	\[g_r := \bigoplus_{|\alpha|+m = r} f_\alpha.\]
	
Note, given a skew cube map $f_*:C_*\too C'_*$, 
	\[\operatorname{Ch}^m(\cone(f_*)) = \cone(\operatorname{Ch}^m(f_*)).\]
A skew cube map $f_*:C_*\too C'_*$ is a \emph{quasi-isomorphism} if $\operatorname{Ch}^m(f_*)$ is a quasi-isomorphism.

We define $\mathfrak C$ to be the subcategory of the $(1+1)$-dimensional cobordism category.

Let $D$ be a link diagram and $\mathcal X$ be the set of crossings of $D$. 
We define $n=|\mathcal X|$, and $n_-$ (resp. $n_+$) as the number of negative (resp. positive) crossings. 
So, $n = n_-+n_+$. 
Each crossing has two possible smoothings, which we will label the 0- and 1-smoothings, defined by Figure~\ref{fig-sm}.
\begin{figure}
	\begin{center}\begin{tikzpicture}
		\nxp{0}{2}
		\hsp{-1.5}{0}
		\vsp{1.5}{0}
		
		\draw [->]  (0,2.2) -- (-.5,1.4);
		\draw [->] (1,2.2) -- (1.5,1.4);
		
		\draw (-1,0) node {0-smoothing};
		\draw (2,0) node {1-smoothing};
	\end{tikzpicture}\end{center}
\caption[0- and 1-smoothings of a crossing.]{\figbld{0- and 1-smoothings of a crossing.}}
\label{fig-sm}
\end{figure}
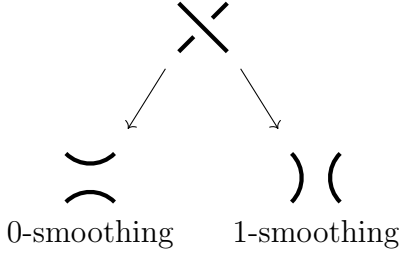

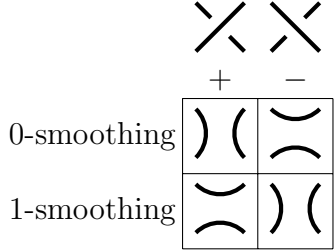
\begin{figure}
	\begin{center}\begin{tikzpicture}
		\pxp{0}{0.2}
		\nxp{1}{0.2}
		
		\hsp{1}{-1.213}
		\vsp{0}{-1.213}
		\hsp{0}{-2.213}
		\vsp{1}{-2.213}
		\draw (0,0) grid (2,-2);
		
		\draw (-1.2,-.5) node {0-smoothing};
		\draw (-1.2,-1.5) node {1-smoothing};
		
		\draw (.5,0.25) node {$+$};
		\draw (1.5,0.25) node {$-$};
	\end{tikzpicture}\end{center}
\caption[0- and 1-smoothings of a crossing in a braid diagram.]{\figbld{0- and 1-smoothings of a crossing in a vertically oriented braid diagram.}}
\label{fig-brsm}
\end{figure}

If each crossing in $D$ is resolved by either a 0- or 1- smoothing, the result is a collection of disjoint circles in the plane, called a \emph{resolution} of $D$. 
There is a resolution cube $R(D)\in\operatorname{CCube}_{\mathcal X}(\mathfrak C)$, where the vertex corresponding to each $\alpha\in\{0,1\}^{\mathcal X}$ is the resolution obtained by replacing each crossing $x\in\mathcal X$ by its $\alpha(x)$-smoothing. The edges correspond to either a merge or a split of a pair of circles as well as the identity cobordism on all other circles. 

We describe the elementary cobordisms pictorially with figures where the source of the morphism is the set of circles at the bottom of the figure, and the target is the set of circles at the top.
See Figure~\ref{fig-cob}. 
To preserve the skew structure later, we fix a sign convention by labeling each crossing with an arrow that induces an arrow on the smoothings of this crossing as in Figure~\ref{fig-signassignment}. There are two possible choices at each crossing. 
\begin{figure}
	\begin{center}\begin{tikzpicture}
		\draw (1.5,-1) node[below] {(a) $M$: the merge cobordism.};
		\botcirc{0}{0}
		\botcirc{2}{0}
		\topcirc{1}{2}
		\draw (0,0) to[out=90, in=270] (1,2);
		\draw (3,0) to[out=90, in=270] (2,2);
		\draw (1,0) to[out=90, in=90] (2,0);
		\nodeb{0}{0}{a_0}
		\nodeb{2}{0}{a_1}
		\nodea{1}{2}{a_0\sim a_1}
	\end{tikzpicture}
	\qquad\qquad
	\begin{tikzpicture}
	\draw (1.5,-1) node[below] {(b) $S$: the split cobordism.};
		\botcirc{1}{0}
		\topcirc{0}{2}
		\topcirc{2}{2}
		\draw (1,0) to[out=90, in=270] (0,2);
		\draw (2,0) to[out=90, in=270] (3,2);
		\draw (1,2) to[out=270, in=270] (2,2);
		\nodeb{1}{0}{a_1}
		\nodea{0}{2}{a_0}
		\nodea{2}{2}{a_1}
	\end{tikzpicture}\end{center}
\caption[Elementary $(1+1)$-Cobordisms.]{\figbld{Elementary $(1+1)$-Cobordisms.}}
\label{fig-cob}
\end{figure}
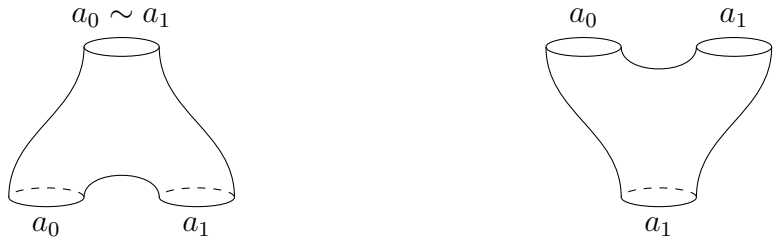

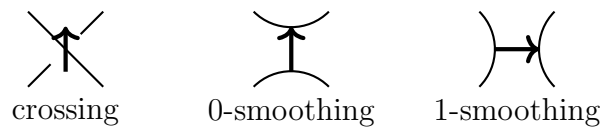
\begin{figure}
\begin{center}
\begin{tikzpicture}
\draw [white] (0,2) circle (.1cm);
\draw[thick] (1,0)--(0,1) (0,0)--(0.3,0.3) (0.7,0.7)--(1,1);
\draw[ultra thick, ->] (.5,.2)--(.5,.8);
\draw (0.5,0) node[below] {crossing};
\draw[thick] (3,0) to[out=45,in=135] (4,0) (3,1) to[out=-45,in=-135] (4,1);
\draw[ultra thick, ->] (3.5,.2)--(3.5,.8);
\draw (3.5,0) node[below] {$0$-smoothing};
\draw[thick] (6,0) to[out=45,in=-45] (6,1) (7,0) to[out=135,in=-135] (7,1);
\draw[ultra thick, <-] (6.8,.5)--(6.2,.5);
\draw (6.5,0) node[below] {$1$-smoothing};
\end{tikzpicture}
\end{center}
\caption[Crossing arrows for orienting cobordisms in the resolution cube.]{\figbld{Crossing arrows for orienting the cobordisms in the resolution cube.} The arrows in a single crossing in the knot diagram (left), its 0-smoothing (middle), and its 1-smoothing (right). The arrow in the middle diagram can also be viewed as the 1-handle attached in the cobordism connecting resolutions, which differ by a 0- and 1-smoothing at this crossing.}
\label{fig-signassignment}
\end{figure}

Each face of the diagram corresponds to one of the four types as depicted in Figure~\ref{fig-sqtype}. 
There is a function
	\[\sgn{\mathcal S}:\mathcal S(R(D)) \too \{\pm 1\},\]
where a square is mapped to $+1$ if it is type C or type Y, and $-1$ if it is type A or type X. 
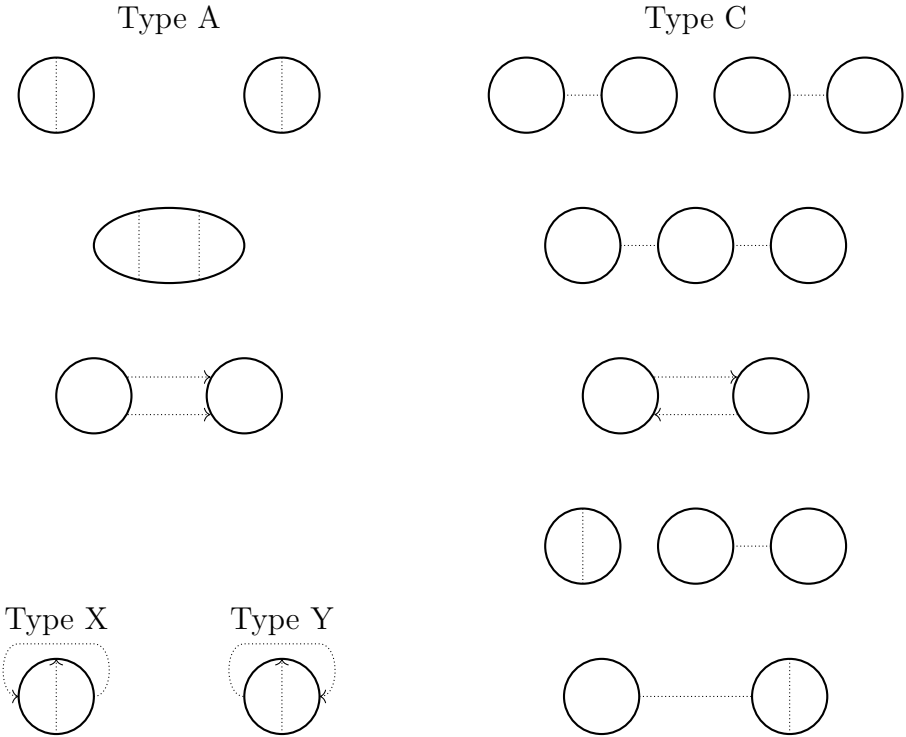
\begin{figure}
\begin{center}
\begin{tikzpicture}

\begin{scope}
\draw (0,0) node {Type A};

\draw[thick] (-1.5,-1) circle (.5);
\draw[densely dotted] (-1.5,-1.5)--(-1.5,-.5);
\draw[thick] (1.5,-1) circle (.5);
\draw[densely dotted] (1.5,-1.5)--(1.5,-.5);

\draw[thick] (0,-3) ellipse (1 and .5);
\draw[densely dotted] (-.4,-3.45)--(-.4,-2.55);
\draw[densely dotted] (.4,-3.45)--(.4,-2.55);

\draw[thick] (-1,-5) circle (.5);
\draw[densely dotted, ->] (-.55,-4.75)--(.55,-4.75);
\draw[densely dotted, ->] (-.55,-5.25)--(.55,-5.25);
\draw[thick] (1,-5) circle (.5);
\end{scope}

\begin{scope}[shift={(7,0)}]
\draw (0,0) node {Type C};

\draw[thick] (-2.25,-1) circle (.5);
\draw[densely dotted] (-1.25,-1)--(-1.75,-1);
\draw[thick] (-.75,-1) circle (.5);
\draw[thick] (.75,-1) circle (.5);
\draw[densely dotted] (1.25,-1)--(1.75,-1);
\draw[thick] (2.25,-1) circle (.5);

\draw[thick] (-1.5,-3) circle (.5);
\draw[densely dotted] (-1,-3)--(-.5,-3);
\draw[thick] (0,-3) circle (.5);
\draw[densely dotted] (1,-3)--(.5,-3);
\draw[thick] (1.5,-3) circle (.5);

\draw[thick] (-1,-5) circle (.5);
\draw[densely dotted, ->] (-.55,-4.75)--(.55,-4.75);
\draw[densely dotted, <-] (-.55,-5.25)--(.55,-5.25);
\draw[thick] (1,-5) circle (.5);

\draw[thick] (-1.5,-7) circle (.5);
\draw[densely dotted] (-1.5,-7.5)--(-1.5,-6.5);
\draw[thick] (0,-7) circle (.5);
\draw[densely dotted] (1,-7)--(.5,-7);
\draw[thick] (1.5,-7) circle (.5);

\draw[thick] (-1.25,-9) circle (.5);
\draw[densely dotted] (-.75,-9)--(.75,-9);
\draw[thick] (1.25,-9) circle (.5);
\draw[densely dotted] (1.25,-9.5)--(1.25,-8.5);
\end{scope}

\begin{scope}[shift={(-1.5,-8)}]
\draw node {Type X};
\draw[thick] (0,-1) circle (.5);
\draw[densely dotted,->] (0,-1.5)--(0,-.5);
\draw[densely dotted,->] (.5,-1) to[out=0,in=0] (.5,-.3)--(-.5,-.3) to[out=180,in=180] (-.5,-1);
\end{scope}

\begin{scope}[shift={(1.5,-8)}]
\draw node {Type Y};
\draw[thick] (0,-1) circle (.5);
\draw[densely dotted,->] (0,-1.5)--(0,-.5);
\draw[densely dotted,<-] (.5,-1) to[out=0,in=0] (.5,-.3)--(-.5,-.3) to[out=180,in=180] (-.5,-1);
\end{scope}
\end{tikzpicture}
\end{center}
\caption[Face types in the resolution cube.]{\figbld{Face types in the resolution cube.} There are four types of faces in the resolution cube depending on the one-handles corresponding to the face's edges. The thicker solid lines represent the relevant circles in the resolution before the cobordisms, and the dotted lines (or arrows) correspond to the one-handles.}\label{fig-sqtype}
\end{figure}

Our construction of the odd Khovanov homology of a link $L$ with diagram $D$ starts with a skew cube $C(D)\in\operatorname{SCube}_{\mathcal X}(\operatorname{Mod}^{gr}_{R})$ where $\operatorname{Mod}^{gr}_{R}$ is the category of graded $R$-modules. 
Each vertex module is the exterior algebra on the free $R$-module generated by $\{v_1,\dots,v_k\}$
\[\exal{v_1,\cdots,v_k}\]
where each generator $v_i$ corresponds with a circle $a_i$ in the corresponding vertex of the resolution cube $R(D)$. 
The construction of the skew cube is inductive on the number of crossings $n$. 
At each stage, for a diagram $D$ with $n$ crossings, we construct a pair of a skew cube $C(D)$ and a function 
	\[\sgn{\mathcal E(C(D))}:\mathcal E(C(D))\too\{\pm1\}\]
with the following properties:
\begin{enumerate}
\item \label{itm-skew} For a face $C(D)$ of type A or X (resp. C or Y) there are an odd (resp. even) number of edges around the face labeled $-1$ by $\sgn{\mathcal E(C(D))}$.
\item \label{itm-signed} The maps in the skew cube $C(D)$ are obtained by multiplying the maps from Formulas~\ref{eq-merge} and \ref{eq-split}, below by $\sgn{\mathcal E(C(D))}(e)$. 
\end{enumerate}

For the base case $|\mathcal E(C(D))|=1$ we define 
	\[\sgn{\mathcal E(C(D))}(e)=1.\]
If the edge $e$ corresponds to the merge cobordism, the corresponding map is defined by
\begin{equation}
\label{eq-merge}
v_0,v_1\mapsto v_0\sim v_1
\end{equation} 
(see Figure~\ref{fig-cob}(a)).
If $e$ corresponds to the split cobordism, the map is defined by
\begin{equation}
\label{eq-split}
1\mapsto (v_0-v_1)
\end{equation} 
(see Figure~\ref{fig-cob}(b)) where the arrow in $\alpha_0$ (as given in Figure~\ref{fig-signassignment}) points from $a_0$ to $a_1$. 

In the inductive step, consider a diagram with $n+1$ crossings. Let $\hat x$ be one of the crossings, and $D_0$ and $D_1$ the diagrams obtained from the 0- and 1-smoothings at $\hat x$. 
By induction, we have a skew cub $C(D_0)$ and a function
	\[\sgn{\mathcal E(C(D_0))}:\mathcal E(C(D_0))\too\{\pm1\},\]
satisfying Properties~\ref{itm-skew} and \ref{itm-signed}. 
Let $\hat C(D_1)$ be the (not necessarily skew) $\mathcal X\backslash\{\hat x\}$-cube where the maps on the edges are defined by (\ref{eq-merge}) if the edge corresponds to a merge cobordism, and (\ref{eq-split}) if the edge corresponds to a split cobordism.
Note, for a face $S\in\mathcal S(\hat C(D_1))$, it is type A (according to Figure~\ref{fig-signassignment}) if the face anti-commutes, type C if it commutes, and type X or Y if it both commutes and anti-commutes (i.e., the composition of two consecutive edges on the face is the zero map).
So $\sgn{\mathcal E}$ is motivated by the need for a sign assignment that guarantees that each face $S\in\mathcal S(C(D_1))$ is skew.
We define the $\mathcal X\backslash\{\hat x\}$-cube $C(D_1)$ to have the same vertices as $\hat C(D_1)$.
For each edge $e_1\in\mathcal E(C(D_1))$, there is a corresponding edge $\hat e_1\in\mathcal E(\hat C(D_1))$ and a corresponding edge $e_0\in\mathcal E(C(D_0))$. 
The pair $e_0$ and $e_1$ correspond to edges in $R(D)$, and there, they specify a unique face $S\in\mathcal S(R(D))$.
We define
	\[e_1=\sgn{\mathcal E(C(D_0))}(e_0)\sgn{\mathcal S}(S)\hat e_1.\]
	
\begin{lem}
Defined as above, $C(D_1)$ is a skew cube.
\end{lem}	

\begin{proof}
Let $S_1\in\mathcal S(\hat C(D_1))$.
There is a corresponding face $S_0\in C(D_0)$, and together, these faces specify a 3-cube in $R(D)$ with $S_0$ as the top face, $S_1$ as the bottom face, and an additional four lateral faces. 
By \cite[Lemma 2.1]{ORSz}, this cube has an even number of faces of type A and X.
We have two cases.
In the first case, $S_0$ and $S_1$ are either both type A or X, or both type C or Y. 
Thus, by \cite[Lemma 2.1]{ORSz}, there are an even number of lateral faces of type A or X. 
Therefore, by the sign assignment of $S_1$, there are an even number of negative signs introduced on the edges of $S_1$.
It follows that the parity of the number of edges that map to $-1$ in $S_0$ matches the parity for $S_1$. 
Since $S_0$ is skew, and $S_1$ is of a matching type, then $S_1$ is skew too.

In the second case, $S_0$ and $S_1$ are not of a matching type. 
That is, one is of type A or X, and the other is of type C or Y. In this case, there is an odd number of lateral faces of type A or X. 
If $S_0$ is type C or Y (and thus $S_1$ is type A or X), by the sign assignment of $S_1$, there are an even number of negative signs introduced on the edges of $S_1$. 
If $S_0$ is type A or X (and thus $S_1$ is type C or Y), by the sign assignment of $S_1$, there are an odd number of negative signs introduced on the edges of $S_1$. 
Therefore, we have the number of negative signs introduced on $S_1$ to guarantee it is skew.
\end{proof}

We define $f^{\hat x}_*:C(D_0)\too C(D_1)$ where for $\alpha\in\{0,1\}^{\mathcal X\backslash\{\hat x\}}$,
$f^{\hat x}_\alpha$ is defined by (\ref{eq-merge}) or (\ref{eq-split}) if the corresponding edge in $R(D)$ is a merge or a split respectively. 
	
\begin{lem}
Defined as above, $f^{\hat x}_*$ is a cube map.
\end{lem}

\begin{proof}
Let $e_0\in\mathcal E(C(D_0))$ with corresponding $e_1\in\mathcal E(C(D_1))$
	\[e_i:C(D_i)_{\alpha_0}\too C(D_i)_{\alpha_1}\]
for $i=0,1$. 
In $R(D)$, there is a unique face $S$ specified by $e_0$ and $e_1$. 
By construction, $e_1$ has the opposite sign assignment of $e_0$ if $S$ is type A or X, and the same sign assignment if $S$ is type C or Y. 
Thus, the definition of $C(D_1)$ guarantees that
	\[f^{\hat x}_{\alpha_1}\circ e_0 = e_1\circ f^{\hat x}_{\alpha_0}.\qedhere\]
\end{proof}

Hence, we define the skew cube
	\[C(D) = \cone(f^{\hat x}_*).\]
This gives us the next skew cube in the induction step.
It remains to be shown that there is a well-defined sign assignment function at this level.

We define
	\[\sgn{\mathcal E(C(D))}:\mathcal E(C(D))\too \{\pm1\}\] 
as follows.
If $e$ corresponds to an edge $e_0\in C(D_0)$, then
	\[\sgn{\mathcal E(C(D))}(e)=\sgn{\mathcal E(C(D_0))}(e_0).\]
If $e$ corresponds to an edge $e_1\in C(D_1)$, $e_0$ is the corresponding edge in $C(D_0)$ and $S$ is the unique square connecting their correspondents in $R(D)$, then
	\[\sgn{\mathcal E(C(D))}(e) = -\sgn{\mathcal E(C(D_0))}(e_0)\sgn{\mathcal S}(S).\]
Note, the negative sign here corresponds to the negative sign in the mapping cone on edges coming from the target of the cone map. 
On all other edges $e$ (the edges that connect $C(D_0)$ to $C(D_1)$),
	\[\sgn{\mathcal E(C(D))}(e) = 1.\]
	
\begin{lem}
With $\sgn{\mathcal E(C(D))}$ defined as above,
for each face $S\in\mathcal S(C(D))$, $S$ has an even number of edges that map to $-1$ if it is type A or X, and an odd number if it is type C or Y. 
\end{lem}

\begin{proof}
Since $C(D)$ is the mapping cone of $f^{\hat x}_*:C(D_0)\too C(D_1)$, where the faces of $C(D_0)$ are already assumed to satisfy this property, and the faces of $C(D_1)$ are constructed to do so, it remains only to look at the faces $S$ which connect an edge $e_0$ that came from $C(D_0)$ to the edge $e_1$ that comes an edge in  $C(D_1)$.
We note the other two edges\----obtained from the mapping cone\----each map to $+1$.

In the first case, let $S$ be type A or X. 
In this case
	\begin{align*}
	\sgn{\mathcal E(C(D_1))}(e_1) &= -\sgn{\mathcal E(C(D_0))}(e_0)\sgn{\mathcal S}(S)\\
	&= \sgn{\mathcal E(C(D_0))}(e_0).
	\end{align*}
Thus, there are exactly zero or two edges that map to $-1$ in $S$.

In the second case, let $S$ be type C or Y.
Then,
	\begin{align*}
	\sgn{\mathcal E(C(D_1))}(e_1) &= -\sgn{\mathcal E(C(D_0))}(e_0)\sgn{\mathcal S}(S)\\
	&= -\sgn{\mathcal E(C(D_0))}(e_0).
	\end{align*}
Thus, there is exactly one edge that maps to $-1$ in $S$. 
\end{proof}

\subsection{The Computer Program}
\label{ssec-prog}
To investigate the invariant, the author has written a suite of modules in Python. 
First, there is the module \texttt{braid.py}, with a \texttt{Braid} class that represents braids by their braid word as a list of signed integers where the absolute value of each integer represents the left strand of the crossing, and the sign of the integer is the sign of the crossing. 
For example, the knot $9_{32}$ is the closure of the braid with braid word
\[\sigma_3^2\sigma_2^{-1}\sigma_3\sigma_2^{-1}\sigma_1\sigma_3\sigma_2^{-1}\sigma_1.\]
In the \texttt{Braid} class, this is represented as the list \texttt{[3,3,-2,3,-2,1,3,-2,1]}. 
Below, we have an interaction running \texttt{python3} on the command line from the folder containing the modules.
\begin{lstlisting}[language=Bash]
>>> import braid
>>> chiral_knot = braid.Braid([3,3,-2,3,-2,1,3,-2,1])
>>> mirror_chiral_knot = chiral_knot.mirror()
>>> reverse_chiral_knot = chiral_knot.reverse()
>>> chiral_knot.self_linking_number()
0
\end{lstlisting}
The class also uses the author's \texttt{braidresolution.sty} package to produce braid diagrams for \LaTeX{} documents. 
Continuing from above, the command below produces \TeX{} for the diagram in Figure~\ref{fig-braidtikz}.
\begin{lstlisting}[language=Bash]
>>> chiral_knot.tex_braid()
\end{lstlisting}
\begin{figure}
\begin{center}\begin{tikzpicture}[scale = 0.4]
	\braidline{4}{3}{0}{0}{1}
	\braidline{4}{3}{0}{1}{1}
	\braidline{4}{2}{0}{2}{0}
	\braidline{4}{3}{0}{3}{1}
	\braidline{4}{2}{0}{4}{0}
	\braidline{4}{1}{0}{5}{1}
	\braidline{4}{3}{0}{6}{1}
	\braidline{4}{2}{0}{7}{0}
	\braidline{4}{1}{0}{8}{1}
\end{tikzpicture}\end{center}
\caption[The diagram of \texttt{chiral{\_}knot} produced by the \texttt{Braid} class.]{\figbld{The diagram of} \texttt{chiral{\_}knot} \figbld{produced by the} \texttt{Braid} \figbld{class.}}
\label{fig-braidtikz}
\end{figure}
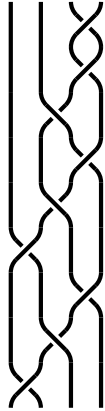

The class also has the addition operator overloaded to compute the connect sum of the closure of two braids. 
\begin{lstlisting}[language=Bash]
>>> right_trefoil = braid.Braid([1,1,1])
>>> left_trefoil = right_trefoil.mirror()
>>> left_trefoil.get_word()
[-1,-1,-1]
>>> connect_sum = right_trefoil + left_trefoil
>>> connect_sum.get_word()
[1,1,1,-2,-2,-2]
\end{lstlisting}

The \texttt{grid.py} module has a \texttt{Grid} class that represents grid diagrams as a pair of lists of integers. 
The first list gives the column index of the X positions row by row, and the second gives the O positions indexed from 0. 
By default, the class creates the braid object for the right-heading braid, but it can do so for any direction in the diagram.
\begin{lstlisting}[language=Bash]
>>> import grid
>>> trefoil_grid = grid.Grid([3,1,0,4,5,2],[0,5,2,1,3,4])
>>> trefoil_braid = trefoil_grid.braid()
>>> trefoil_braid.get_word()
[-1, -2, 1, 2, 2, -1]
>>> trefoil_leftward_braid = trefoil_grid.braid('left')
>>> trefoil_leftward_braid.get_word()
[1, -2, 1, -2]
\end{lstlisting}

The \texttt{Grid} class also has methods to produce \TeX{} for the grid diagram, (e.g., \texttt{trefoil{\_}braid.tex{\_}grid()}, see Figure~\ref{fig-gridtikz}(a)), the corresponding knot diagram (e.g., \texttt{trefoil{\_}braid.tex{\_}knot()}, see Figure~\ref{fig-gridtikz}(b)), and the Legendrian front of the grid (e.g., \texttt{trefoil{\_}braid.tex{\_}Legendrian{\_}front()}, see Figure~\ref{fig-gridtikz}(c)).
\begin{figure}
\begin{center}
\begin{tikzpicture}[scale=0.75]
\begin{scope}[shift = {(-0.5,0.5)}]
\draw (0,0) grid (6,6);
\end{scope}
\draw [white] (2.5,8) node{.};
\draw (0,6) circle (.3cm);
\draw (3,6) ++(-.3,-.3) -- ++ (.6,.6) ++(-.6,0) --++(.6,-.6);
\draw (5,5) circle (.3cm);
\draw (1,5) ++(-.3,-.3) -- ++ (.6,.6) ++(-.6,0) --++(.6,-.6);
\draw (2,4) circle (.3cm);
\draw (0,4) ++(-.3,-.3) -- ++ (.6,.6) ++(-.6,0) --++(.6,-.6);
\draw (1,3) circle (.3cm);
\draw (4,3) ++(-.3,-.3) -- ++ (.6,.6) ++(-.6,0) --++(.6,-.6);
\draw (3,2) circle (.3cm);
\draw (5,2) ++(-.3,-.3) -- ++ (.6,.6) ++(-.6,0) --++(.6,-.6);
\draw (4,1) circle (.3cm);
\draw (2,1) ++(-.3,-.3) -- ++ (.6,.6) ++(-.6,0) --++(.6,-.6);
\draw (2.5,-1) node {(a) the grid diagram.};
\end{tikzpicture}
\qquad
\begin{tikzpicture}[scale=0.75]
\draw [white] (2.5,8) node{.};
\draw (0,6) circle (.3cm);
\draw (3,6) ++(-.3,-.3) -- ++ (.6,.6) ++(-.6,0) --++(.6,-.6);
\draw[->] (0,6) ++ (1*0.5,0) --++(2,0);
\draw (5,5) circle (.3cm);
\draw (1,5) ++(-.3,-.3) -- ++ (.6,.6) ++(-.6,0) --++(.6,-.6);
\draw[->] (5,5) ++ (-1*0.5,0) --++(-3,0);
\draw (2,4) circle (.3cm);
\draw (0,4) ++(-.3,-.3) -- ++ (.6,.6) ++(-.6,0) --++(.6,-.6);
\draw[->] (2,4) ++ (-1*0.5,0) --++(-1,0);
\draw (1,3) circle (.3cm);
\draw (4,3) ++(-.3,-.3) -- ++ (.6,.6) ++(-.6,0) --++(.6,-.6);
\draw[->] (1,3) ++ (1*0.5,0) --++(2,0);
\draw (3,2) circle (.3cm);
\draw (5,2) ++(-.3,-.3) -- ++ (.6,.6) ++(-.6,0) --++(.6,-.6);
\draw[->] (3,2) ++ (1*0.5,0) --++(1,0);
\draw (4,1) circle (.3cm);
\draw (2,1) ++(-.3,-.3) -- ++ (.6,.6) ++(-.6,0) --++(.6,-.6);
\draw[->] (4,1) ++ (-1*0.5,0) --++(-1,0);
\draw[white, line width = 5pt] (0,4) ++ (0,1*0.5) --++ (0,1);
\draw[->] (0,4) ++ (0,1*0.5) --++(0,1);
\draw[white, line width = 5pt] (1,5) ++ (0,-1*0.5) --++ (0,-1);
\draw[->] (1,5) ++ (0,-1*0.5) --++(0,-1);
\draw[white, line width = 5pt] (2,1) ++ (0,1*0.5) --++ (0,2);
\draw[->] (2,1) ++ (0,1*0.5) --++(0,2);
\draw[white, line width = 5pt] (3,6) ++ (0,-1*0.5) --++ (0,-3);
\draw[->] (3,6) ++ (0,-1*0.5) --++(0,-3);
\draw[white, line width = 5pt] (4,3) ++ (0,-1*0.5) --++ (0,-1);
\draw[->] (4,3) ++ (0,-1*0.5) --++(0,-1);
\draw[white, line width = 5pt] (5,2) ++ (0,1*0.5) --++ (0,2);
\draw[->] (5,2) ++ (0,1*0.5) --++(0,2);
\draw (2.5,-1) node {(b) the knot diagram.};
\end{tikzpicture}
\qquad
\begin{tikzpicture}[scale=0.75]
\draw [white] (2.5,6) node{.};
\draw (0,0) to[out=0, in=180] (3,3);
\draw (2,0) to[out=0, in=180] (6,4);
\draw (2,-2) to[out=0, in=180] (4,0);
\draw (4,-2) to[out=0, in=180] (7,1);
\draw (7,-1) to[out=0, in=180] (9,1);
\draw (7,-3) to[out=0, in=180] (9,-1);
\draw[white, line width = 5pt] (0,0) to[out=0, in=180] (2,-2);
\draw (0,0) to[out=0, in=180] (2,-2);
\draw[white, line width = 5pt] (2,0) to[out=0, in=180] (4,-2);
\draw (2,0) to[out=0, in=180] (4,-2);
\draw[white, line width = 5pt] (4,0) to[out=0, in=180] (7,-3);
\draw (4,0) to[out=0, in=180] (7,-3);
\draw[white, line width = 5pt] (3,3) to[out=0, in=180] (7,-1);
\draw (3,3) to[out=0, in=180] (7,-1);
\draw[white, line width = 5pt] (7,1) to[out=0, in=180] (9,-1);
\draw (7,1) to[out=0, in=180] (9,-1);
\draw[white, line width = 5pt] (6,4) to[out=0, in=180] (9,1);
\draw (6,4) to[out=0, in=180] (9,1);
\draw (4.5,-4) node {(c) the Legendrian front.};
\end{tikzpicture}
\end{center}
\caption[The diagrams of \texttt{chiral{\_}knot} produced by the \texttt{Grid} class.]{\figbld{The diagrams of} \texttt{chiral{\_}knot} \figbld{produced by the} \texttt{Grid} \figbld{class.}}
\label{fig-gridtikz}
\end{figure}

Corresponding to a diagram with $n$ indexed crossings, there is a resolution cube with $2^n$ vertices. 
In the program, these vertices are represented as an integer whose binary representation has $i^{\text{th}}$ digit $b_i$. 
Thus, a vertex then has the resolution with the $b_i$-smoothing at the $i^{\text{th}}$ crossing for all $i$, and data corresponding to the vertices is stored in lists whose indices correspond directly to vertices. 
For the additional information in the resolution cube, we have the \texttt{cube.py} module, which contains three classes:  the \texttt{EdgeStruct} class, the \texttt{SquareStruct} class, and the \texttt{Vertical} class. 

The \texttt{EdgeStuct} class is a container for the edges of the resolution cube. 
In memory, edges are stored as a pair of integers: the vertex of the start of the edge as described above, and the crossing index (indexed from 0), which changes from $0$ to $1$. 
This container is used to store both the maps between vertices and the sign assignment, which makes the faces anticommutative. 
An \texttt{EdgeStruct} object also provides an iterator that iterates through the edges primarily in order of the vertices and secondarily in order of the crossing index.

The \texttt{SquareStruct} class is a container for the faces (or squares) of the resolution cube. 
In memory, squares are stored as a triple of integers: the leftmost vertex in the square, and the two integers representing the two edges in the square that adjoin the leftmost vertex. 
In the program, there is one \texttt{SquareStruct} object that is used as an iterator, but it also acts as a container that stores information about the commutativity type of the square. 
The order in which it iterates through the squares is fundamentally different from the \texttt{EdgeStuct} class. 
Since any $(n+1)$-dimensional cube can be represented as two $n$-dimensional subcubes (one shifted over to the right one spot) connected by $2^n$ additional edges, the edges' signs are computed inductively from squares on the subcubes. 
The iterator provided by an \texttt{EdgeStruct} objector encodes this inductive order.

The \texttt{Vertical} class provides an iterator through all of the generators of the vertex modules that sum to each module in the chain complex. 
This class is used to produce the matrices from which the homology and the invariant are computed. 

There is also the \texttt{khovanovhomology.py} module, which contains the container class: \texttt{KhovanovHomology}. 
As a container, it stores the homological information. Furthermore, it contains a collection of the methods necessary to compute it. 
The methods include those which can compute the even and odd Khovanov homologies over $\Z$ and any field, as well as Plamenevskaya's invariant and its odd analog defined in Section~\ref{sec-inv}. 
As even and odd Khovanov homology are categorifications of the Jones polynomial, in the process of computing the resolutions, the \texttt{KhovanovHomology} class can calculate the Jones polynomial. 

To support these, there are two more modules: \texttt{fields.py} and \texttt{algebra.py}. 
The \texttt{fields.py} module contains a class \text{FE} that handles elements in $\Q$ or $\Z/p$. 
The benefit of having a single class for field elements of all characteristics is that only a single number needs to be changed to compute the even and odd Khovanov homologies and invariants over different fields. 
The \texttt{algebra.py} module contains a variety of basic supporting functions as well as implementations of different algorithms necessary for computational homology. 
For an integer matrix $A$, there is a function that computes its Smith normal form: a matrix $D$ that is zero except on its diagonal and $D_{i,i}$ divides $D_{i+1,i+1}$, as well as the unimodular matrices $S$ and $T$ such that 
	\[SAT = D.\]
These output can be used for the inputs of another function that finds the smallest positive $n$ such that 
	\[Ax=ny\]
if such an $n$ exists. 
The former function, along with another function that handles row reduction over a field, provide the computations necessary to compute the homology. 
The latter is used to compute if the invariant is zero and if it is torsion. 

In the next example, we show how to generate the odd Khovanov homology for a knot for its grid diagram. 
Below, we compute the odd Khovanov homology for $8_{19}$.
\begin{lstlisting}[language=Bash]
>>> import grid
>>> G = grid.Grid([0,1,6,2,5,7,8,3,4,9],[6,7,8,9,1,4,5,0,2,3])
>>> G.comp_full_graded_homology()
>>> B = G.braid()
>>> B.comp_full_graded_homology()
KH'_( 0)(L) = Z^1[ 7] + Z^1[ 5]
KH'_( 1)(L) = 0
KH'_( 2)(L) = Z^1[11] + Z^1[ 9]
KH'_( 3)(L) = 0
KH'_( 4)(L) = Z/2[13] + Z/2[11]
KH'_( 5)(L) = Z^1[17] + (Z^1 + Z/3)[15] + Z/3[13]
Wide knot, sigma = 6, sl = 5.
\end{lstlisting}

While computing the invariant, the code checks if the invariant is zero in homology, and if not, if it is torsion.
Below, we have the computation that shows that the invariant does not distinguish the pair of knots in \cite[$m10_{140}$]{BM2}
\begin{lstlisting}[language=Bash]
>>> import grid
>>> L1 = grid.Grid([8,7,1,3,5,4,2,6,0],[3,2,4,6,8,7,0,1,5])
>>> B1 = L1.braid()
>>> B1.comp_inv()
Inv NonZero
\end{lstlisting}

\begin{lstlisting}[language=Bash]
>>> L2 = grid.Grid([8,7,0,3,5,4,6,1,2],[3,1,4,6,8,7,2,5,0])
>>> B2 = L2.braid()
>>> B2.comp_inv()
Inv NonZero
\end{lstlisting}

\subsection{Computational Observations}
\label{ssec-compob}
If $\sigma(K)$ is the signature of a knot, then in the knots that have been examined so far, the knots in which the invariant is nonzero correspond exactly with knots in which 
	\[sl(K) =\sigma(K)-1.\]
As seen before, if $K$ is alternating, then
	\[sl(K)\leq\sigma(K)-1,\]
thus the invariant is nonzero in the cases where this maximum is reached. 
This is supported by Proposition~\ref{prop-negstab}, which implies of $\psi(K)\neq 0$ then $K$ is not the negative stabilization of another knot. 
If it were, there would be a knot $K'$ which had the same topological knot type as $K$, but $sl(K') = sl(K) + 2$. 
	
The even and odd Plamenevskaya invariants are zero and nonzero in the same knots for knots examined.
 
 If $n$ is the length of the braid used in the computation (the number of crossings), and $n_-$ is the number of negative crossings, the invariant is usually zero if $n_-/n > 0.25$, and usually nonzero if $n_-/n < 0.25$. 
 There are a limited number of exceptions, namely the following, which are zero,
\begin{alignat*}
99_{11}:&\quad[3, 3, 3, 3, -2, 1, 3, -2, 1]&&n_-/n = 0.\overline{2},\\
m9_{20}:&\quad[3, 3, 3, -2, 1, 3, -2, 1, 1]&&n_-/n = 0.\overline{2},\\
\intertext{and this one, which is nonzero,}
m9_{35}:&\quad[4, 4, 3, -4, 3, 3, 2, 1, -3, -3, -2, 1, 3, 2]\quad&&n_-/n =0.\overline{285714}.
\end{alignat*}
Note, for all braids computed, at least one of $\psi(B)$ or $\psi(mB)$ is zero. 

So far, the invariant has not been shown to be effective. 
The Plamenevskaya invariant has not been shown to be effective either, and there is also evidence that it might not be.
However, among the reasons why an odd analog of Plamenevskaya's invariant could be effective even if Plamenevskaya's invariant is not is the construction of odd Khovanov homology.
Ozsv\'ath and Szab\'o constructed a spectral sequence whose $E_2$ term is $KH(L;\Z/2)$, which converges to $\widehat{HF}(L)$ \cite{OSz}. 
Attempts to lift the spectral sequence to $\Z$ failed, but inspired the definition of the odd Khovanov homology, where it is conjectured that there is a spectral sequence whose $E_2$ term is $KH'(L;\Z)$ that converges to $\widehat{HF}(\Sigma(mL))$. 
Ng, Ozsv\'ath and Thurston showed the filtered homotopy type of $\widehat{HF}$, called knot Floer homology, could be used to distinguish pairs of transverse knots with the same classical transverse invariants. 
Tracing the Plamenevskaya invariant through to the knot Floer homology is limited by the $\Z/2$ coefficients, however the even analogue is not. 
We can also compare this to a similar spectral sequence from odd Khovanov homology to the framed instanton homology of the branched double cover of a link \cite{Sca}
 and the contact invariant in instanton Floer homology \cite{BS2}.

\newpage
\bibliography{oddkhoho}
\bibliographystyle{alpha}
\end{document}